\documentclass{amsart}

\usepackage{a4}
\usepackage{amsmath} 
\usepackage{amssymb}
\usepackage[all,2cell,emtex]{xy}
\input xypic       
\SilentMatrices

\SelectTips{cm}{10}                               

\theoremstyle{plain} 
\newtheorem*{thm'}{Theorem}                    
\newtheorem*{corol}{Corollary}
\newtheorem{thm}{Theorem}[section]
\newtheorem{prop}[thm]{Proposition}
\newtheorem{lemme}[thm]{Lemma}
\newtheorem{cor}[thm]{Corollary}
\theoremstyle{remark}                                             
\newtheorem{rem}[thm]{Remark}
\newtheorem{exemple}[thm]{Example}
\newtheorem{exemples}[thm]{Examples}
\theoremstyle{definition}                                         
\newtheorem{definition}[thm]{Definition}
\newtheorem*{defin}{Definition}
\newtheorem{paragr}[thm]{}

\newcommand{\cat}{{\mathcal{C} \mspace{-2.mu} \it{at}}}
\newcommand{\Hom}{\operatorname{\mathsf{Hom}}}
\newcommand{\sHom}{\operatorname{\underline{\mathsf{Hom}}}}
\newcommand{\op}[1]{{#1}^{\circ}}
\newcommand{\Int}{\textstyle{\int}}
\newcommand{\ob}{\operatorname{\mathsf{Ob}}}
\newcommand{\W}{\mathcal{W}}
\newcommand{\Winf}{\W_{\infty}}
\newcommand{\fl}{\operatorname{\mathsf{Fl}}}
\newcommand{\hot}{\mathsf{Hot}}
\newcommand{\smp}[1]{ \varDelta_{#1}}
\newcommand{\e}{\varepsilon}

\renewcommand{\longmapsto}{{\hskip -2.5pt\xymatrixcolsep{1.3pc}\xymatrix{\ar@{|->}[r]&}\hskip -2.5pt}}
\renewcommand{\mapsto}{{\hskip -2.5pt\xymatrixcolsep{.9pc}\xymatrix{\ar@{|->}[r]&}\hskip -2.5pt}}
\newcommand{\toto}{{\hskip -2.5pt\xymatrixcolsep{1.3pc}\xymatrix{\ar[r]&}\hskip -2.5pt}}
\renewcommand{\to}{{\hskip -2.5pt\xymatrixcolsep{1pc}\xymatrix{\ar[r]&}\hskip -2.5pt}}
\renewcommand{\hookrightarrow}{{\hskip -1.5pt\raise 1.5pt\vbox{\xymatrixcolsep{.9pc}\xymatrix{\ar@{^{(}->}[r]&}}\hskip -3.5pt}}

\newcommand{\Cofibr}[1]{\theta^{}_{#1}}
\newcommand{\ind}[2]{\widetilde{#1}}

\newcommand{\ctlh}[1]{#1^{}_!}
\newcommand{\ctlhprim}[1]{#1'_!}

\newcommand{\Asph}{\mathcal{A}}

\newcommand{\FAsph}{\mathcal{A}sph}
\newcommand{\FAS}[1]{\FAsph(#1)}
\newcommand{\FAsphinv}{\FAsph^{-1}}
\newcommand{\FASinv}[1]{\FAS{#1}^{-1}}
\newcommand{\FonctAsph}{\FAsph^{}_{\hskip -2pt\Asph}}
\newcommand{\FonctAS}[1]{\FonctAsph(#1)}

\newcommand{\CHot}{\operatorname{\mathsf{Hot}}}
\newcommand{\CHOT}[1]{\CHot(#1)}
\newcommand{\CatHot}{\CHot^{}_{\Asph}}
\newcommand{\CatHOT}[1]{\CatHot(#1)}
\newcommand{\CatAHmtp}{{\,\overline{\cat}\,}}

\newcommand{\imdir}[1]{#1^{}_!}
\newcommand{\imdirh}[1]{#1^{}_!}

\newcommand{\Dbb}{\mathbb{D}}

\title[Asphericity structures, smooth functors, and fibrations]{Asphericity structures, smooth functors, \break and fibrations}
\author{Georges Maltsiniotis}

\address{Institut de Math\'ematiques de Jussieu\\
Universit\'e Paris 7 Denis Diderot\\
\hfill\break\indent Case Postale~7012\\
2, place Jussieu\\
F-75251, PARIS cedex 05\\
FRANCE}
\email{maltsin@math.jussieu.fr}
\urladdr{http://www.math.jussieu.fr/\raise -3.3pt\vbox{\hbox{$\widetilde{ \ }\,$}}maltsin/}

\dedicatory{This paper was translated from French by Jonathan Chiche.}

\begin{document}

\begin{abstract}
The aim of this paper is to generalize Grothendieck's theory of
smooth functors in order to include within this framework the theory of
fibered categories. We obtain in particular a new characterization of fibered categories.
\end{abstract}

\maketitle

\section*{Introduction}

In ``Pursuing Stacks''~\cite{PS}, Grothendieck defines the notions of
proper functors and smooth functors by analogy with
the cohomological properties of proper morphisms and smooth morphisms between schemes.
If
$$
\mathcal{D}\,=\quad
\raise 20pt\vbox{
\xymatrixrowsep{1.8pc}
\xymatrix{
X'\ar[r]^{g}\ar[d]_(.47){f'}
&X\ar[d]^{f}
\\
Y'\ar[r]_{h}
&Y
}
}
\qquad\qquad
$$
is a commutative square of schemes and $M$ is an abelian \'etale sheaf on $X$,
there is a canonical ``base change'' morphism
$$h^*Rf_*(M)\toto Rf'_*\,g^*(M)\quad.$$
Proper and smooth base change theorems
state~\cite[lectures 12, 13 and 16]{SGA4}
that if the square $\mathcal{D}$ is cartesian,
then this morphism is an isomorphism in the following two cases:
\begin{itemize}
\item[(\emph{a})] $f$ is proper and $M$ is of torsion;
\item[(\emph{b})] $h$ is smooth, $M$ is of torsion prime to the
residual exponents of $Y$, and $f$ is quasi-compact and quasi-separated.
\end{itemize}
\medbreak

We denote by $\cat$ the category of small categories and by $\Winf$ the
class of arrows in $\cat$ whose image under the nerve functor
is a simplicial weak equivalence, \emph{i.e.}~a morphism of simplicial sets 
whose topological realization is a
homotopy equivalence. The localized category $\W^{-1}_\infty\cat$,
obtained by formally inverting the arrows which belong to $\Winf$, 
is then equivalent to the homotopy category $\hot$ of CW\nobreakdash-complexes.
For every small category $I$, define $\Dbb(I)$ by
$$\Dbb(I)=\bigl(\W^{\op{I}}_\infty\bigr)^{-1}\sHom(\op{I},\cat)\quad,$$
where $\sHom(\op{I},\cat)$ denotes the category of presheaves on
$I$ with values in $\cat$, \emph{i.e.}~the contravariant functors
from $I$ to $\cat$. Let
$\W^{\op{I}}_\infty$ be the class of natural transformations between functors
from $\op{I}$ to $\cat$ which are componentwise in $\Winf$:
$$\W^{\op{I}}_\infty=\{\alpha\in\fl(\sHom(\op{I},\cat))\,|\,\alpha_i\in\Winf,i\in\ob(I)\}\quad.$$
For every arrow $u:I\to J$ of $\cat$, the inverse image functor
$$\sHom(\op{J},\cat)\toto\sHom(\op{I},\cat)\quad,\qquad
F\longmapsto F\circ\op{u}\quad,$$
defines a functor
$$u^*:\Dbb(J)\toto\Dbb(I)\quad.$$
By a result of Alex Heller~\cite{Hel}, this functor has
a left and a right adjoint
$$u^{}_!:\Dbb(I)\toto\Dbb(J)\qquad\hbox{and}\qquad u^{}_*:\Dbb(I)\toto\Dbb(J)$$
respectively.
\medbreak

Grothendieck defines the notion of proper (resp.~smooth) functor
as follows.

\begin{defin}
A functor between small categories $u:A\to B$ (resp.~$w:B'\to B$) is said
to be \emph{proper} (resp.~\emph{smooth}) if for every cartesian square
$$
\xymatrix{
&A'\ar[r]^{v}\ar[d]_(.45){u'}
&A\ar[d]^{u}
\\
&B'\ar[r]_{w}
&B
&\hskip -20pt,
}$$
the base change morphism
$$w^*u^{}_*\toto u'_*v^*\quad$$
(or, equivalently, 
$$v^{}_!u'{}^*\toto u^*w^{}_!\quad,$$
transpose of the previous morphism) is an isomorphism, and this property
remains true after any base change. 
\end{defin}

Grothendieck obtains simple characterizations of proper
functors and smooth functors, and he notices that his theory
relies only on a small number of formal properties of $\Winf$,
the most important one being Quillen's Theorem~A~\cite{Qu}. This
leads him to introduce the notion of basic localizer.
\medbreak

A \emph{basic localizer} is a class $\W$ of arrows
of $\cat$ satisfying the following properties.
\begin{itemize}
\item[\textsf{Loc1}] (Weak saturation.) \emph{a}) Identities belong to the class $\W$.

\emph{b}) If two out of the three arrows in a commutative triangle
belong to $\W$, then so does the third.

\emph{c}) If $i$ is an arrow of $\cat$ which has
a retraction $r$ and if $ir$ belongs to $\W$, then $i$ belongs to $\W$.
\item[\textsf{Loc2}] If $A$ is a small category which has a final object, then
the functor $A\to e$ from $A$ to the final category $e$ belongs to~$\W$.
\item[\textsf{Loc3}] (Relative Quillen's Theorem A.)
For every commutative triangle 
$$
\xymatrixcolsep{1pc}
\xymatrix{
A\ar[rr]^{u}\ar[dr]_{v}
&&B\ar[dl]^{w}
\\
&C
}$$
of $\cat$, 
if for every object $c$ of $C$, the functor $u/c:A/c\to B/c$
(where $A/c= A\times_C(C/c)$ and $B/c=B\times_C(C/c)$ stand for the
``comma categories'' of objects over~$c$)
induced by $u$ belongs to $\W$, then so does $u$.
\end{itemize}
\medbreak

Grothendieck conjectures in~\cite{PS} that
$\Winf$ is the smallest basic localizer.
This conjecture is proved by D.\nobreakdash-C.~Cisinski in~\cite{C2}.
\medbreak

A small category $A$ is said to be $\W$\nobreakdash-aspheric if the
morphism $A\to e$ belongs to $\W$. A functor between small categories 
$A\to B$ is said to be $\W$\nobreakdash-aspheric if for every object $b$ of
$B$, the category $A/b$ is $\W$\nobreakdash-aspheric. It follows
from \textsf{Loc3} that a $\W$\nobreakdash-aspheric functor belongs to 
$\W$. As in the case $\W=\Winf$, we define, for every
small category $I$, the category $\Dbb(I)$ by
$$\Dbb(I)=\bigl(\W^{\op{I}}\bigr)^{-1}\sHom(\op{I},\cat)\quad,$$
and for every arrow $u:I\to J$ of $\cat$, we can show by elementary
methods~\cite{Ma} that the inverse image functor $u^*:\Dbb(J)\to\Dbb(I)$
has a left adjoint 
\hbox{$u^{}_!:\Dbb(I)\to\Dbb(J)$}. We can therefore define the notion 
of $\W$\nobreakdash-proper and $\W$\nobreakdash-smooth functors as in the case of $\W=\Winf$, 
using base change morphisms related to $u^{}_!$.
\medbreak

In fact, the theory developed in D.\nobreakdash-C.~Cisinski's thesis~\cite{C1}
shows also the existence of a right adjoint $u^{}_*:\Dbb(I)\to\Dbb(J)$,
assuming a mild hypothesis on $\W$, namely that $\W$ is
\emph{accessible}, \emph{i.e.}~that it is
generated by a \emph{set} of arrows of $\cat$ (it is the smallest
basic localizer containing a \emph{set} of arrows
of $\cat$). 
\medbreak

Grothendieck obtains the following characterization of $\W$\nobreakdash-smooth functors~\cite{Der},~\cite{Ma}.

\begin{thm'}
Let $u:A\to B$ be a morphism of $\cat$. The following conditions
are equivalent:
\begin{itemize}
\item[(a)] $u$ is $\W$\nobreakdash-smooth, \emph{i.e.}~for every diagram of cartesian squares
$$
\xymatrix{
&A''\ar[r]^{v}\ar[d]_{u''}  
&A'\ar[r]\ar[d]^{u'}
&A\ar[d]^{u}
\\
&B''\ar[r]_{w}
&B'\ar[r]
&B
&\hskip -20pt,
}$$
the base change morphism $u''_!v^*\to w^*u'_!$ is an isomorphism;
\item[(b)] for every object $b$ of $B$, the inclusion
$$A_b\toto b\backslash A\quad,\qquad a\longmapsto (a,1_b:b\to u(a)=b)\quad$$
(where $A_b$ denotes the fiber of $A$ over $b$ and $b\backslash A=A\times_B(b\backslash B)$
the ``comma category'' of objects under $b$) is a $\W$\nobreakdash-aspheric functor.
\item[(c)] for every diagram of cartesian squares
$$
\xymatrix{
&A''\ar[r]^{v}\ar[d]
&A'\ar[r]\ar[d]
&A\ar[d]^{u}
\\
&B''\ar[r]_{w}
&B'\ar[r]
&B
&\hskip -20pt,
}$$
if the functor $w$ is $\W$\nobreakdash-aspheric, then so is $v$. 
\end{itemize}
\end{thm'}

Grothendieck is filled with wonder by the following fact that he deduces from this theorem:

\begin{corol}
A morphism $u:A\to B$ of $\cat$ is $\W$\nobreakdash-smooth if and only if the
opposite functor $\op{u}:\op{A}\to\op{B}$ is $\W$\nobreakdash-proper.
\end{corol}

This result is a straightforward consequence of the characterization above,
of the ``dual'' characterization of $\W$\nobreakdash-proper morphisms,
and of the fact that a functor between small categories belongs to $\W$
if and only if the opposite functor does~\cite{PS},~\cite{Der},~\cite{Ma}.
\medbreak

Fibrations (morphisms $u:A\to B$ of $\cat$ such that $A$ is a fibered category
over~$B$) are important examples of $\W$\nobreakdash-smooth functors. 
This follows from the fact that if $u:A\to B$ is a fibration,
then for every object $b$ of $B$ the inclusion
$$A_b\toto b\backslash A\quad,\qquad a\longmapsto (a,1_b:b\to u(a)=b)\quad$$
has a right adjoint, and from the fact that a functor between
small categories which has a right adjoint is a $\W$\nobreakdash-aspheric functor.
\medbreak

The impetus for this work was the observation that the class of
fibrations shares many formal properties with the class
of $\W$\nobreakdash-smooth functors and that the class
of functors which have a right adjoint shares many properties with the class of $\W$\nobreakdash-aspheric functors. 
Nevertheless, there is no basic localizer $\W$ such that
the $\W$\nobreakdash-smooth morphisms are exactly the fibrations, or such that
the $\W$\nobreakdash-aspheric morphisms are exactly the functors
which have a right adjoint. Moreover, we notice that
the notions of $\W$\nobreakdash-aspheric, of $\W$\nobreakdash-smooth and of $\W$\nobreakdash-proper morphisms
depend on the sole class of $\W$\nobreakdash-aspheric categories.
If two basic localizers have the same class of aspheric objects,
then the corresponding notions of aspheric functors,
smooth functors and proper functors are the same. 
So I tried to look for the minimal properties that
a class of \emph{objects} of $\cat$ should fulfil 
in order to give rise to
a theory of smooth functors. 
This has led me, at the cost of
breaking the symmetry of passing to the opposite category, to
introduce the notion of right asphericity structure. Since this
notion is not self-dual, one needs also to consider the
dual notion of left asphericity structure, giving rise to a theory of proper functors. 
\medbreak

In the first section, we define the notion of right asphericity structure, 
implying
a notion of aspheric functor. We notice
that there exists a minimal right asphericity structure, defined by
the class of small categories which have a final object, 
the corresponding aspheric functors 
being exactly the functors between small categories
which have a right adjoint. 
For every basic localizer, we define a right asphericity structure with the same aspheric functors as the given localizer.
\medbreak

The second and the third sections are logically independent.
In the second one, we consider localizations of categories of functors
with domain a small category and codomain $\cat$, with respect to
natural transformations that are componentwise aspheric, and we prove the existence
of left homotopical Kan extensions.
\medbreak

In the third one, we introduce the notion of smooth functor associated with a
right asphericity structure. In this framework, it splits into two notions:
smooth functors and weakly smooth functors. We give several equivalent characterizations
for each of these two notions and we study their main properties.
We show that the smooth functors with respect to the minimal right asphericity structure
are exactly the fibrations, and that the weakly smooth functors with respect to this structure are exactly the prefibrations.
This result provides a new characterization of fibered categories.
Finally, we observe that,  when the right asphericity structure is defined by a basic
localizer, the notions of smooth functor and weakly smooth functor are equivalent and equivalent to 
the notion of smooth functor with respect to the basic localizer. 
\medbreak

In the last section, we combine the results of the two previous ones. We prove a characterization
of smooth functors (with respect to a right asphericity structure) in terms
of 
base change morphisms. As in this paper we consider
\emph{covariant} functors with values in $\cat$, rather than presheaves, we obtain a characterization which is dual to Grothendieck's formulation.
\medbreak

Recently, Jonathan Chiche has generalized the notion of right asphericity structure, and some of the results of this paper, to the framework of $2$\nobreakdash-categories~\cite{Chiche}.

\section{Right asphericity structures}

\begin{paragr}
A \emph{right asphericity structure} is a class
$\Asph$ of small categories satisfying the following two conditions.
\begin{itemize}
\item[As1] Every small category which has a final object is in $\Asph$.
\item[As2] For every functor between small categories $u:A\to B$, if $B$ is in $\Asph$, 
and if for every object $b$ of $B$, $A/b$ is in $\Asph$, then $A$ is also in $\Asph$.
\end{itemize}
\end{paragr}

\begin{exemple}\label{3strasphmin}
The class $\Asph$ of small categories which have a final object 
is a right asphericity structure. Indeed, condition As1
being fulfilled by definition, it is enough to check 
condition As2. Let $B$ be a category with a final object $b_0$,
 and $u:A\to B$ a morphism in $\cat$ such
that for every object $b$ of $B$, $A/b$ has a final object.
Then the category $A$, which is isomorphic to $A/b_0$, has a final
object, hence the assertion.
This right asphericity structure is the
\emph{minimal right asphericity structure}.
\end{exemple}

\begin{exemple}\label{3strasphlocfond}
Let $\W$ be a basic localizer.
The class of $\W$\nobreakdash-aspheric categories is a right asphericity structure.
\end{exemple}
\goodbreak

\noindent
\emph{In the sequel, we fix, once and for all, a right asphericity structure~$\Asph$.}

\begin{paragr}\label{3defcatasph}
A small category $C$ is said to be $\Asph$\nobreakdash-\emph{aspheric}, or
more simply \emph{aspheric}, if it belongs to $\Asph$.
Using this definition, condition As1 states that a small category which has
a final object is aspheric.
\end{paragr}

\begin{prop}\label{3prodcatasph}
The product of two small aspheric categories is
an aspheric category.
\end{prop}

\begin{proof}
Let $A$ and $B$ be two small aspheric categories. Let us
show that their product $A\times B$ is aspheric too. 
Considering the first projection $A\times B\to A$, it suffices by~As2 
to show that for every object $a$ of $A$, the category
$(A\times B)/a\simeq A/a\times B$ is aspheric. Considering
the second projection $A/a\times B\to B$, it suffices by
As2 to show that for every object $b$ of $B$, the category 
$(A/a\times B)/b\simeq A/a\times B/b$ is aspheric. Now, the latter
has a final object, and the assertion follows from As1.
\end{proof}

\begin{paragr}\label{3deffonctasph}
A morphism $u:A\to B$ is said to be $\Asph$\nobreakdash-\emph{aspheric}, or
more simply \emph{aspheric}, if for every object $b$ of $B$, the category $A/b$ is aspheric. 
In terms of this definition, condition~As2 states that
if the codomain of an aspheric morphism is an aspheric category,
then its domain is aspheric, too. Notice that a small
category $A$ is aspheric if and only if the
functor $A\to e$ from $A$ to the final category is an
aspheric morphism. For every small category $A$, the identity functor
$1_A:A\to A$ is aspheric. Indeed, for every object $a$ of $A$, the category
$A/a$ has a final object, and the assertion follows from condition~As1.
\end{paragr}

\begin{exemple}\label{3carfonctasphmin}
If $\Asph$ is the minimal right asphericity structure (\emph{cf.}~\ref{3strasphmin}),
then the aspheric morphisms are exactly the functors
between small categories which have a right adjoint.
Indeed, a functor between small categories $u:A\to B$ has
a right adjoint if and only if, for every object $b$ of
$B$, the category $A/b$ has a final object.
\end{exemple}

\begin{exemple}
If $\Asph$ is the right asphericity structure associated with a basic localizer $\W$
(\emph{cf.}~\ref{3strasphlocfond}), then the $\Asph$\nobreakdash-aspheric morphisms are exactly the
$\W$\nobreakdash-aspheric functors.
\end{exemple}

\begin{cor}\label{3prodfonctasph}
Let $u:A\to B$ and $u':A'\to B'$ be two aspheric morphisms
of $\cat$. Then the functor $u\times u':A\times A'\to B\times B'$
is aspheric.
\end{cor}

\begin{proof}
We have to show that
for every object $(b,b')$ of $B\times B'$, the category\penalty -10000{}
$(A\times A')/(b,b')$ is aspheric. Now, $(A\times A')/(b,b')$
is canonically isomorphic to $(A/b)\times(A'/b')$ and by hypothesis
$A/b$ and $A'/b'$ are aspheric. Therefore the assertion follows from
proposition~\ref{3prodcatasph}.
\end{proof}

\begin{prop}\label{3asphlocbase}
Let
$$\xymatrixcolsep{.8pc}
\xymatrix{
A\ar[rr]^{u}\ar[dr]
&&B\ar[dl]^{v}
\\
&C
}$$
be a commutative triangle in $\cat$. The morphism $u$ is
aspheric if and only if, for every object $c$ of $C$,
the morphism $u/c:A/c\to B/c$, induced by $u$, is aspheric.
\end{prop}

\begin{proof}
We check immediately that for every object $c$ of $C$ and for every object
$(b,p:v(b)\to c)$ of $B/c$, the category $(A/c)/(b,p)$ is
canonically isomorphic to $A/b$. We deduce that if $u$ is
aspheric, then so is $u/c$. Conversely, assume that for
every object $c$ of $C$ the morphism $u/c$ is aspheric. Then,
for every object $b$ of $B$, the category
$(A/v(b))/(b,1_{v(b)})\simeq A/b$ is aspheric, which proves
that $u$ is aspheric.
\end{proof}

\begin{prop}\label{3compfonctasph}
Let $\raise 3.5pt\vbox{\xymatrixcolsep{1pc}\xymatrix{A\ar[r]^u&B\ar[r]^v&C}}$
be a pair of composable morphisms of $\cat$. If $u$ and $v$ are aspheric,
then so is $vu$.
\end{prop}

\begin{proof}
For every object $c$ of $C$, since $v$ is aspheric, the
category $B/c$ is aspheric, and by the previous proposition,
since $u$ is aspheric, so is the functor
$u/c:A/c\to B/c$ induced by $u$. From this we deduce that the category
$A/c$ is aspheric (As2), which proves that $vu$ is aspheric. 
\end{proof}

\begin{prop}\label{3adjgasph}
Let $u:A\to B$ be a morphism of $\cat$. If $u$ has a right adjoint, 
then $u$ is aspheric.
\end{prop}

\begin{proof}
Let $v:B\to A$ be a right adjoint to $u$. The functorial bijection
$$\Hom_B(u(a),b))\simeq\Hom_A(a,v(b)),\quad a\in\ob(A),\ b\in\ob(B),\quad$$
implies that, for every object $b$ of $B$, the category $A/b$ is isomorphic
to the category $A/v(b)$, which has a final object. It follows that $A/b$ is
aspheric, hence the proposition.
\end{proof}

\begin{cor}\label{3eqcatasph}
An equivalence between small categories is an aspheric functor.
\end{cor}

\begin{paragr}\label{3catcof}
Let $u:A\to B$ be a functor. We recall that the {\it fiber\/} of $u$ over
an object $b$ of $B$ is the subcategory (not full in general) $A_b$ of $A$ whose objects 
are the objects $a$ of $A$ such that $u(a)=b$, and whose morphisms 
are the arrows $f$ of $A$ such that \hbox{$u(f)=1_b$}. An arrow $c:a\to a'$
of $A$ is called {\it cocartesian\/}
(with respect to $u$, or over $B$) if for every morphism
$f:a\to a''$ of $A$ such that \hbox{$u(f)=u(c)$}, there is a unique morphism $g:a'\to a''$
of $A$ \smash{such that $u(g)=1_{u(a')}$ and $f=gc$.}
$$\xymatrixrowsep{1.8pc}
\xymatrixcolsep{4.pc}
\xymatrix{
&a''
\\
a\ar[ur]^f\ar[r]_c
&a'\ar@{-->}[u]_g
\\
u(a)\ar[r]_{u(c)}
&u(a')
}$$
The arrow $c$ is called \emph{hypercocartesian} (with respect to $u$, or over $B$)
if for every morphism $f:a\to a''$ of $A$ and every morphism $h:u(a')\to u(a'')$ 
of $B$ such that $u(f)=h\hskip 1pt u(c)$, there is a unique morphism $g:a'\to a''$
of $A$ such that $u(g)=h$ and $f=gc$.
$$\xymatrixrowsep{1.8pc}
\xymatrixcolsep{2.5pc}
\xymatrix{
&&a''
\\
a\ar[urr]^f\ar[r]_c
&a'\ar@{-->}[ur]_g
\\
u(a)\ar[r]_{u(c)}
&u(a')\ar[r]_h
&u(a'')
}$$
The functor $u$ is called a {\it precofibration\/} if for every morphism
$p:b\to b'$ of $B$, and for every object $a$ of $A$ over $b$
(\emph{i.e.}~$u(a)=b$),
there is a cocartesian morphism $c:a\to a'$ over $p$ 
(\emph{i.e.}~$u(c)=p$).
The functor $u$ is called a {\it cofibration\/} if $u$ is a precofibration,
and if the class of cocartesian morphisms of $A$ is stable under composition. 
It is easily checked that $u$ is a cofibration
if and only if for every morphism
$p:b\to b'$ of~$B$ and for every object $a$ of $A$ over $b$,
there is a hypercocartesian morphism $c:a\to a'$ over $p$.
\smallbreak

Dually, a morphism of $A$ is called {\it cartesian\/} (resp.~\emph{hypercartesian})
with respect to $u$, or over $B$, if the corresponding morphism of $A^\circ$
(the opposite category of $A$) is cocartesian (resp.~hypercocartesian) with respect to
$u^\circ:A^\circ\to B^\circ$. The functor $u$ is called a {\it prefibration\/}
(resp.~a {\it fibration\/}) if the functor $u^\circ$ is a precofibration
(resp.~a cofibration).
\end{paragr}

\begin{lemme}\label{3carprecof}
{\it A functor $u:A\to B$ is a precofibration if and only if
for every object $b$ of $B$, the canonical functor $A_b\to A/b$, which sends
an object $a$ of the fiber $A_b$ of $u$ over $b$ to the object
$(a,1_b)$ of $A/b$, has a left adjoint.}
\end{lemme}

The proof is left to the reader. 

\begin{prop}\label{3precoffibasph}
{\it Let $u:A\to B$ be a morphism of $\cat$, and assume that $u$
is a precofibration and that for every object $b$ of $B$,
the fiber $A_b$ of $u$ over $b$ is aspheric. Then $u$
is aspheric.}
\end{prop}

\begin{proof}
According to the previous lemma, for every object $b$ of $B$, the functor
$$\begin{aligned}
i_b:A_b&\toto A/b\quad\\
\noalign{\vskip 1pt}
a&\mapsto(a,1_b)
\end{aligned}$$
has a left adjoint
$$j^{}_b:A/b\toto A_b\quad.$$
It follows from proposition~\ref{3adjgasph} that the functor $j^{}_b$ is aspheric.
Since $A_b$ is aspheric, so is
 $A/b\,$ (As2), hence $u$ is aspheric.
\end{proof}

\begin{paragr}\label{3deffonctlocasph}
A morphism $u:A\to B$ of $\cat$ is said to be \emph{locally $\Asph$\nobreakdash-aspheric}, or
more simply \emph{locally aspheric}, if for every object $a$
of $A$, the morphism
$$A/a\toto B/b\quad,\qquad b=u(a)\quad,$$
induced by $u$, is aspheric.
\end{paragr}

\begin{exemples}\label{3exfonctlocasph}
\emph{a}) A fully faithful aspheric functor is locally aspheric.
\smallbreak

\emph{b}) For any small category $C$, the functor $C\to e$
is locally aspheric.
\end{exemples}

\begin{prop}\label{3prop1fonctlocasph}
\emph{a)} A morphism  $u:A\to B$  of $\cat$
is locally aspheric if and only if, for every
object $b$ of $B$, the functor \hbox{$u/b:A/b\to B/b$}, induced by $u$,
is aspheric.
\smallbreak

\emph{b)} Let $\xymatrixcolsep{1pc}\xymatrix{A\ar[r]^(.46)u&B\ar[r]^(.46)v&C}$
be a pair of composable morphisms of $\cat$. If $u$ and $v$ are locally aspheric,
then so is $vu$.
\smallbreak

\emph{c)} Let $u:A\to B$, $u':A'\to B'$ be two locally aspheric morphisms 
of $\cat$. Then the functor $u\times u':A\times A'\to B\times B'$
is locally aspheric.
\end{prop}

\begin{proof}
The first assertion is straightforward, the second one is a consequence of
proposition~\ref{3compfonctasph}, and the third one is a consequence of corollary 
\ref{3prodfonctasph}.
\end{proof}

\section{Homotopical Kan extensions}

\noindent
\emph{In this paragraph, we fix, once and for all, 
a right asphericity structure $\Asph$.}

\begin{paragr}\label{3defhot}
We denote by $\FonctAsph$, or simply 
$\FAsph$, the class of
aspheric morphisms of $\cat$, and we set
$$\CHot=\CatHot=\FAsphinv\cat\quad.$$
More generally, for every small category $I$, we denote by
$\FonctAS{I}$, or simply 
$\FAS{I}$, the class of
morphisms of $\sHom(I,\cat)$ which are componentwise
aspheric, and we set
$$\CHOT{I}=\CatHOT{I}=\FASinv{I}\sHom(I,\cat)\quad.$$
For every arrow $w:J\to I$ of $\cat$, if we denote by
$$w^*:\sHom(I,\cat)\toto\sHom(J,\cat)$$
the inverse image functor, we have
$$w^*(\FAS{I})\subset\FAS{J}\quad.$$
Therefore $w^*$ induces a functor between the localized categories,
denoted 
$$w^*:\CHOT{I}\toto\CHOT{J}\quad,$$
too, such that the following square is commutative:
$$
\xymatrix{
&\sHom(I,\cat)\ar[r]^{w^*}\ar[d]_{\gamma^{}_I}
&\sHom(J,\cat)\ar[d]^{\gamma^{}_J}
\\
&\CHOT{I}\ar[r]_{w^*}
&\CHOT{J}
&\hskip -20pt,
}$$
where the vertical arrows are the canonical localization functors.
\end{paragr}

\begin{exemple}\label{3hotmin}
If $\Asph$ is the minimal right asphericity structure (\emph{cf}.~\ref{3strasphmin}),
then $\CHot$ is the category $\CatAHmtp$ of small categories
up to homotopy, \emph{i.e.}~the category with same objects as
$\cat$, and such that for every pair of small categories
$A$, $B$, 
$$\Hom^{}_\CatAHmtp(A,B)=\pi_0\,\sHom(A,B)\quad.$$
Indeed, let us denote by
$$\gamma:\cat\toto\CHot\quad,\qquad Q:\cat\toto\CatAHmtp$$
the canonical functors. By the universal properties of 
these functors, it is sufficient to show that:
\begin{itemize}
\item[\emph{a})] the image under $Q$ of an aspheric functor is a homotopism;
\item[\emph{b})] the images under $\gamma$ of two homotopic functors are equal.
\end{itemize}
In order to show the first assertion, we notice that, since
aspheric morphisms of $\cat$ with respect to the minimal right
asphericity structure are exactly the functors between small categories
which have a right adjoint (\ref{3carfonctasphmin}), they are indeed homotopisms.
In order to show the second assertion, it suffices to show that for
every morphism \hbox{$h:\smp{1}\times A\to B$} of $\cat$, where
$\smp{1}$ stands for the category $\{0\to1\}$, we have $\gamma(h_0)=\gamma(h_1)$,
where $h_\e=h(\partial_\e\times1_A)$, $\e=0,1$, and \hbox{$\partial_\e:e\to\smp{1}$}
is the morphism from the final category $e$ to $\smp{1}$ defined by
the object $\e$ of $\smp{1}$. Now, if we denote by
$p:\smp{1}\times A\to A$ the second projection, we have 
$p(\partial_0\times1_A)=1_A=p(\partial_1\times1_A)$. Since the category
$\smp{1}$ has a final object, it
follows from As1 and corollary~\ref{3prodfonctasph} that the functor $p$
is aspheric. 
Therefore $\gamma(p)$ is an isomorphism of $\CHot$, and the equalities
$$\gamma(p)\gamma(\partial_0\times1_A)=\gamma(p(\partial_0\times1_A))=
 \gamma(p(\partial_1\times1_A))=\gamma(p)\gamma(\partial_1\times1_A)\quad$$
imply the relation
$$\gamma(\partial_0\times1_A)=\gamma(\partial_1\times1_A)\quad,$$
and hence we have
$$\gamma(h_0)=\gamma(h(\partial_0\times1_A))=\gamma(h)\gamma(\partial_0\times1_A)=
 \gamma(h)\gamma(\partial_1\times1_A)=\gamma(h(\partial_1\times1_A))=\gamma(h_1)\ .$$
\end{exemple}

\begin{exemple}\label{3hotlocfond}
If $\Asph$ is the right asphericity structure associated with
a basic localizer $\W$ (\emph{cf}.~\ref{3strasphlocfond}), then it follows from the theory
developed by \hbox{D.-C.~Cisinski~\cite{C1}} that
$$\CHot\simeq\W_{\mathit{\hskip -2pt asph}}^{-1}\,\cat\quad,$$
where $\W_{\mathit{asph}}$ stands for the basic localizer
generated by the arrows $A\to e$ of $\cat$, for $A$ a $\W$\nobreakdash-aspheric category.
Indeed, let us denote by
$$
\xymatrix{
\cat\ar[r]^(.3){\gamma}
&\CHot\quad,
\qquad\cat\ar[r]^(.57){\gamma'}
&\W_{\mathit{\hskip -2pt asph}}^{-1}\,\cat
}$$
the canonical functors. By their universal properties, 
it suffices to show that:
\begin{itemize}
\item[\emph{a})] $\FAsph\subset\W_{\mathit{asph}}$;
\item[\emph{b})] $\gamma(\W_{\mathit{asph}})$ is contained in the
class of isomorphisms of $\CHot$.
\end{itemize}
The assertion (\emph{a}) is obvious. In order to prove (\emph{b}), we
first observe that 
every basic localizer being a filtered union of accessible ones,
an easy argument shows that we can 
assume that the basic localizer is accessible.  
Thomason-Cisinski theory~\cite{C1},~\cite{Th} then states that there exists
a proper closed model category structure on $\cat$ whose
weak equivalences are the elements of $\W_{\mathit{asph}}$.
This implies that $\cat$ has the structure of a category of fibrant objects 
\cite{Br}, with weak equivalences the elements
of $\W_{\mathit{asph}}$ and fibrations the arrows
$u:A\to B$ of $\cat$ such that for every diagram of cartesian squares
$$
\xymatrix{
&A''\ar[r]^{v}\ar[d]
&A'\ar[r]\ar[d]
&A\ar[d]^{u}
\\
&B''\ar[r]_{w}
&B'\ar[r]
&B
&\hskip -20pt,
}$$
if $w$ is in $\W_{\mathit{asph}}$, then so is $v$.
Hence, by Ken Brown's lemma~\cite[I.1 Factorization lemma]{Br},
in order to prove that the elements of $\gamma(\W_{\mathit{asph}})$
are isomorphisms of $\CHot$, it suffices to show that
if such a fibration is in $\W_{\mathit{asph}}$, then its image under $\gamma$
is an isomorphism. Now, such an arrow is universally
in $\W_{\mathit{asph}}$. It
is therefore a $\W_{\mathit{asph}}$\nobreakdash-aspheric morphism, and in particular
$\W$\nobreakdash-aspheric, which proves the assertion.
\smallbreak

We can easily generalize this argument to show that for
every small category $I$, the category $\CHOT{I}$ is isomorphic
to the localization of the category $\sHom(I,\cat)$ by the arrows
which are componentwise in $\W_{\mathit{asph}}$.
\end{exemple}

\begin{paragr}\label{3defasphrel}
For every object $I$ of $\cat$, we denote by $\cat/I$ the category
of small categories over $I$, whose objects are the pairs $(A,A\to I)$,
consisting of a small category $A$ and of a functor $A\to I$, and whose morphisms
are commutative triangles
$$
\xymatrixcolsep{1.2pc}
\xymatrix{
A\ar[rd]\ar[rr]
&&A'\ar[ld]
\\
&I
&.
}$$
We denote by $\FAsph/I$ the class of arrows
$$
\xymatrixcolsep{1.2pc}
\xymatrix{
A\ar[rd]\ar[rr]
&&A'\ar[ld]
\\
&I
}$$
of $\cat/I$ such that $A\to A'$ is an aspheric functor.
For every arrow $w:J\to I$ of $\cat$, we denote by
$$\cat/w:\cat/J\toto\cat/I\quad$$
the functor defined by
$$
\xymatrixrowsep{1pc}
\xymatrixcolsep{1pc}
\xymatrix{
A\ar[dd]_{v}
&&&
A\ar[dd]_{wv}
\\
&\ar@{|->}[r]
&\hskip 5pt 
\\
J
&&&I
&.
}$$
We have
$$(\cat/w)(\FAsph/J)\subset\FAsph/I\quad,$$
and thus the functor $\cat/w$ induces a functor
$$\overline{\cat/w}:(\FAsph/J)^{-1}(\cat/J)\toto(\FAsph/I)^{-1}(\cat/I)\quad.$$
\end{paragr}

\begin{paragr}\label{3propadjGr}
For every small category $I$, we define functors
$$\begin{aligned}
&{\xymatrixcolsep{2.7pc}\xymatrix{\cat/I\ar[r]^(.4){\Theta_I}&\sHom(I,\cat)}}
\quad,\qquad
{\xymatrixcolsep{2.7pc}\xymatrix{\sHom(I,\cat)\ar[r]^(.6){\Theta'_I}&\cat/I}}
\cr
\noalign{\vskip 5pt}
&\kern -15pt(A,\,A\toto I)\kern 2pt\longmapsto\kern 2pt(i\longmapsto A/i)
\quad\kern 2pt,
\kern 18mm F\kern 6pt\longmapsto\kern 6pt(\textstyle\int F,\,\int F\toto I)\quad,
\end{aligned}$$
where $\int F$ stands for Grothendieck's construction of the
cofibered category over $I$ defined by the functor $F$, and $\int F\to I$ is
the canonical functor. We recall that the objects of the category
$\int F$ are the pairs $(i,a)$, where $i$ is an object of $I$ and
$a$ is an object of $F(i)$. A morphism of $\int F$ from
$(i,a)$ to $(i',a')$ is a pair $(k,f)$, where $k:i\to i'$
is an arrow of $I$ and $f:F(k)(a)\to a'$ is an arrow of $F(i')$.
The composition in $\int F$ is defined, for a pair of composable morphisms 
$$
\xymatrixcolsep{2.5pc}
\xymatrix{
(i,a)\ar[r]^(.45){(k,f)}
&(i',a')\ar[r]^(.425){(k',f')}
&(i'',a'')\quad,
}$$
by the formula
$$(k',f')\circ(k,f)=(k'k,f'\cdot F(k')(f))\quad.$$
The canonical functor $\int F\to I$, which is a cofibration, is defined by
$$(i,a)\longmapsto i\quad,\qquad (k,f)\longmapsto k\quad.$$
For every object $i$ of $I$, the fiber over
$i$ of the functor $\int F\to I$ is canonically isomorphic to the category $F(i)$.
\smallbreak

For every morphism $w:J\to I$ of $\cat$,  
the functors 
$${\xymatrixcolsep{4.3pc}\xymatrix{\cat/J\ar[r]^(.43){\Theta_I\circ\cat/w}&\sHom(I,\cat)}}
\quad,\qquad
\xymatrixcolsep{4.pc}\xymatrix{\sHom(I,\cat)\ar[r]^(.56){\Theta'_J\circ w^*}&\cat/J
}$$
form an adjoint pair with counit and unit
$$\e:\Theta_I\circ\cat/w\circ\Theta'_J\circ w^*\toto
1_{\sHom(I,\cat)}\ ,\quad
\eta:1_{\cat/J}\toto\Theta'_J\circ w^*\circ
\Theta_I\circ\cat/w\quad$$
defined as follows.
\smallbreak

\emph{a}) \textit{Definition of $\e$.} For every functor $F:I\to \cat$,
and every object $i$ of $I$, we need to define a functor
$$\e^{}_{F,i}:\big(\Int Fw\bigr)\bigm/i\toto F(i)\quad.$$
The objects of the category $(\int Fw)/i$ are the triples
$(j,a,p:w(j)\to i)$, where $j$ is an object of $J$, $a$ is an object of $Fw(j)$,
and $p$ is an arrow of $I$. A morphism of $(\int Fw)/i$ from
$(j,a,p:w(j)\to i)$ to $(j',a',p':w(j')\to i)$ is a pair
$(l,f)$, where $l:j\to j'$ is an arrow of $J$ and $f:Fw(l)(a)\to a'$
is an arrow of $Fw(j')$ such that
$$
\raise 25pt\vbox{
\xymatrixcolsep{1pc}
\xymatrix{
w(j)\ar[rr]^{w(l)}\ar[rd]_{p}
&&w(j')\ar[ld]^{p'}
\\
&i
}
}
\kern 2cm
p=p'w(l)
\quad.
$$
We define the functor $\e^{}_{F,i}$ by
$$\begin{aligned}
&{\e^{}_{F,i}(j,a,p:w(j)\to i)=F(p)(a)\quad,}\cr
\noalign{\vskip 3pt}
&{\e^{}_{F,i}(l,f)=F(p')(f):F(p')Fw(l)(a)=F(p)(a)\toto F(p')(a')\quad.}
\end{aligned}$$
We leave it to the reader to check compatibility
with respect to composition and identities, as well as functoriality
in $i$ and in $F$.
\smallbreak

\emph{b}) \textit{Definition of $\eta$.} For every object $(A, v:A\to J)$ of $\cat/J$,
we need to define a functor
$$\eta^{}_{(A,v)}:A\toto\Int A/w(j)\quad,$$
over $J$,
where $\int A/w(j)$ denotes, by a slight abuse, the cofibered category
over $J$ defined by the functor
$$J\toto\cat\quad,\qquad j\longmapsto A/w(j)\quad.$$
The objects of $\int A/w(j)$ are the triples $(j,a,p:wv(a)\to w(j))$, where $j$
is an object of $J$, $a$ is an object of $A$ and $p$ is an arrow of $I$.
A morphism
$$(j,a,p:wv(a)\to w(j))\toto(j',a',p':wv(a')\to w(j'))\quad$$
is a pair $(l,f)$, where $l:j\to j'$ is an arrow of $J$ and
$f:a\to a'$ is an arrow of $A$ such that the square
$$
\xymatrixcolsep{3pc}
\xymatrix{
wv(a)\ar[r]^{wv(f)}\ar[d]_{p}
&wv(a')\ar[d]^{p'}
\\
w(j)\ar[r]_{w(l)}
&w(j')
}
$$
is commutative. For every object $a$ of $A$, we define
$$\eta^{}_{(A,v)}(a)=\bigl(v(a),a,1_{wv(a)}:wv(a)\to wv(a)\bigr)\quad,$$
and for every arrow $f:a\to a'$ of $A$,
$$\eta^{}_{(A,v)}(f)=\bigl(v(f),f\bigr):
\bigl(v(a),a,1_{wv(a)}
\bigr)\toto
\bigl(v(a'),a',1_{wv(a')}
\bigr).$$
Compatibility with respect to composition and identities,
as well as functoriality in~$(A,v)$, are easily checked.
\smallbreak

We leave it to the reader to check the triangle identities
$$\bigl((\Theta'_J\circ w^*)\star\varepsilon\bigr)
  \bigl(\eta\star(\Theta'_J\circ w^*)\bigr)\kern -.5pt=\kern -.5pt1_{\Theta'_J\circ w^*}
  \, ,\ 
  \bigl(\varepsilon\star(\Theta_I\circ\cat/w)\bigr)
  \bigl((\Theta_I\circ\cat/w)\star\eta\bigr)\kern -.5pt=\kern -.5pt1_{\Theta_I\circ\cat/w}
$$
$$
\xymatrixcolsep{3.7pc}
\xymatrix{
\Theta'_J\circ w^*\ar[r]^(.30){\eta\star(\Theta'_J\circ w^*)}
&\Theta'_J\circ w^*\circ\Theta_I\circ\cat/w\circ\Theta'_J\circ w^*\ar[r]^(.67){(\Theta'_J\circ w^*)\star\varepsilon}
&\Theta'_J\circ w^*
}$$
$$
\xymatrixcolsep{4.3pc}
\xymatrix{
\Theta_I\kern -.5pt\circ\kern -.5pt\cat/w\ar[r]^(.33){(\Theta_I\circ\cat/w)\star\eta}
&\Theta_I\kern -.5pt\circ\kern -.5pt\cat/w\kern -.5pt\circ\kern -.5pt\Theta'_J\kern -.5pt\circ\kern -.5pt w^*\kern -.5pt\circ\kern -.5pt\Theta_I\kern -.5pt\circ\kern -.5pt\cat/w\ar[r]^(.66){\varepsilon\star(\Theta_I\circ\cat/w)}
&\Theta_I\kern -.5pt\circ\kern -.5pt\cat/w
}$$
which prove that $(\Theta_I\circ\cat/w,\,\Theta'_J\circ w^*)$ 
is indeed an adjoint pair.
\smallbreak

As a special case, for $J=I$ and $w=1_I$, we get that
$(\Theta_I,\Theta'_I)$ is an adjoint pair.
\end{paragr}

\begin{lemme}\label{3lemmeclef}
Let $I$ be a small category, $F,G:I\to\cat$ be two functors
and \hbox{$u:F\to G$} be a natural transformation. If, for every object $i$
of $I$, the functor \hbox{$u_i:F(i)\to G(i)$} is aspheric, then the functor
$$\Int u:\int F\toto\Int G\quad$$
is aspheric.
\end{lemme}

\begin{proof}
We define a functor
$$H:\Int G\toto\cat\quad$$
as follows. For every object $(i,b)$ of $\int G$, $i\in\ob(I)$, $b\in\ob(G(i))$,
we set
$$H(i,b)=F(i)/b\quad,$$ 
and for every morphism
$(k,g):(i,b)\to(i',b')$ of $\int G$, where $k:i\to i'$ is an
arrow of $I$ and $g:G(k)(b)\to b'$ is an arrow of $G(i')$, we define 
$$H(k,g):F(i)/b\toto F(i')/b'\quad$$
by
$$(a,\,p:u_i(a)\to b)\longmapsto (F(k)(a),\,g\cdot G(k)(p):u_{i'}F(k)(a)\to b')\quad.$$

$$
\xymatrixcolsep{.3pc}
\xymatrixrowsep{.5pc}
\xymatrix{
u_{i'}F(k)(a)\ar@{=}[d]\ar[rr]
&&b'
\\
G(k)u_i(a)\ar[rdd]_{G(k)(p)}
\\
\\
&G(k)(b)\ar[ruuu]_{g}
}$$
\bigbreak

\noindent
Let us consider the category $\int H$. The objects of this category are
the quadruples $(i,\,b,\,a,\,p:u_i(a)\to b)$, $i\in \ob(I)$, $b\in\ob(G(i))$, $a\in\ob(F(i))$,
$p\in\fl(G(i))$. A~morphism from $(i,b,a,p)$ to
$(i',b',a',p')$ is a triple $(k,g,f)$
$$k:i\to i'\,\in\,\fl(I)\ ,\ \ \ g:G(k)(b)\to
b'\,\in\,\fl(G(i'))\ ,\ \ \ f:F(k)(a)\to a'\,\in\,\fl(F(i'))$$
such that the square
$$
\xymatrixrowsep{.5pc}
\xymatrix{
u_{i'}F(k)(a)\ar@{=}[d]\ar[rr]^(.54){u_{i'}(f)}
&&u_{i'}(a')\ar[dddd]^{p'}
\\
G(k)u_i(a)\ar[ddd]_{G(k)(p)}
\\
\\
\\
G(k)(b)\ar[rr]_(.55){g}
&&b'
}$$
is commutative. For every pair of composable morphisms
$$
\xymatrixcolsep{3.8pc}
\xymatrix{
(i,\,b,\,a,\,p)\ar[r]^(.46){(k,\,g,\,f)}
&(i',\,b',\,a',\,p')\ar[r]^(.46){(k',\,g',\,f')}
&(i'',\,b'',\,a'',\,p'')
}\quad,$$
the composite morphism is defined by
$$(k',\,g',\,f')\circ(k,\,g,\,f)=(k'k,\,g'\cdot G(k')(g),\,f'\cdot F(k')(f))\quad.$$
The canonical functor $\Cofibr{H}:\int H\to\int G$ is defined by
$$(i,b,a,p)\mapsto (i,b)\, ,\ (i,b,a,p)\,\in\,\ob(\Int H)\, ,
\quad (k,g,f)\mapsto (k,g)\, ,\ (k,g,f)\,\in\,\fl(\Int H)\, .$$
We shall define a functor $S:\int F\to\int H$ such that the triangle
$$
\xymatrixcolsep{1.8pc}
\xymatrix{
&\int H\ar[rd]^{\Cofibr{H}}
\\
\int F\ar[ru]^{S}\ar[rr]_{\int u}
&&\int G
}$$
is commutative. For every object $(i,a)$ of $\int F$,
$i\in\ob(I)$, $a\in\ob(F(i))$, we set 
$$S(i,a)=(i,\,u_i(a),\,a,\,1_{u_i(a)})\quad,$$
and for every morphism $(k,f):(i,a)\to(i',a')$ of $\int F$,
where $k:i\to i'$ is an arrow of $I$ and $f:F(k)(a)\to a'$ is an
arrow of $F(i')$,
$$S(k,f)=(k,\,u_{i'}(f),\,f):(i,\,u_i(a),\,a,\,1_{u_i(a)})\toto(i',\,u_{i'}(a'),\,a',\,1_{u_{i'}(a')})\quad.$$
It is straightforward to check the compatibility of $S$ with identities and
composition, as well as to check the commutativity of the triangle above. We
shall define a functor $R:\int H\to\int F$ which will be a right adjoint
to $S$ as well as a retraction to it. For every object $(i,b,a,p)$ of
$\int H$, we set $R(i,b,a,p)=(i,a)$, and for every morphism
$(k,g,f):(i,b,a,p)\to(i',b',a',p')$ of $\int H$, we set $R(k,g,f)=(k,f)$.
Compatibility of $R$ with identities and composition is
obvious, and so is the equality $RS=1_{\int F}$. We define a
natural transformation $\e:SR\to1_{\int H}$ as follows.
For every object $(i,\,b,\,a,\,p:u_i(a)\to b)$ of $\int H$, we set
$$\e_{(i,b,a,p)}=(1_i,p,1_a):SR(i,b,a,p)=S(i,a)=(i,\,u_i(a),\,a,\,1_{u_i(a)})\toto(i,b,a,p)\quad,$$
and we notice immediately that we have thus defined a morphism of $\int H$. 
For every morphism $(k,g,f):(i,b,a,p)\to(i',b',a',p')$ of $\int H$, the square
$$
\xymatrixcolsep{4pc}
\xymatrixrowsep{2.5pc}
\xymatrix{
SR(i,b,a,p)\ar[r]^{\e_{(i,b,a,p)}}\ar[d]_{SR(k,g,f)}
&(i,b,a,p)\ar[d]^{(k,g,f)}
\\
SR(i',b',a',p')\ar[r]^{\e_{(i',b',a',p')}}
&(i',b',a',p')
}$$
is commutative, as expected. Indeed,
$$\begin{aligned}
(k,g,f)\e_{(i,b,a,p)}&=(k,g,f)(1_i,p,1_a)=(k,g\,G(k)(p),f)\quad,\cr
\noalign{\vskip 3pt}
\e_{(i',b',a',p')}SR(k,g,f)&=(1_{i'},p',1_{a'})S(k,f)\cr
&=(1_{i'},p',1_{a'})(k,u_{i'}(f),f)=(k,p'u_{i'}(f),f)
\end{aligned}$$
and $p'u_{i'}(f)=g\,G(k)(p)$, since $(k,g,f)$ is a morphism of $\int H$.
From this we deduce that $\e:SR\to1_{\int H}$ is indeed a natural transformation.
Finally, for every object $(i,b,a,p)$ of $\int H$, we have
$$R(\e_{(i,b,a,p)})=R(1_i,p,1_a)=(1_i,1_a)\quad,$$
and for every object $(i,a)$ of $\int F$, we have
$$\e_{S(i,a)}=\e_{(i,u_{i}(a),a,1_{u_{i}(a)})}=(1_i,1_{u_{i}(a)},1_a)\quad,$$
which proves that 
$$\e:SR\toto1_{\int H}\qquad\hbox{and}\qquad 1_{1_{\int F}}:1_{\int F}\toto RS=1_{\int F}\quad$$
satisfy the triangle identities and that $R$ is a right adjoint to $S$.
Therefore, it follows from proposition~\ref{3adjgasph} that $S$ is an aspheric functor.
Since, for every object~$i$ of $I$, $u_i$ is an
aspheric functor by hypothesis, $\Cofibr{H}$ is a cofibration whose fibers 
are aspheric. From this we deduce that the functor $\Cofibr{H}$ is
aspheric (proposition~\ref{3precoffibasph}),
and therefore that so is the composite morphism $\int u=\Cofibr{H}S$
(proposition~\ref{3compfonctasph}), which finishes the proof.
\end{proof}

\begin{thm}\label{3theqcatGr}
For every small category $I$, we have
$$\FAsph/I=\Theta^{-1}_I(\FAS{I})\quad,\qquad\Theta'_I(\FAS{I})\subset\FAsph/I\quad$$
and the functors
$$\overline{\Theta}_I:(\FAsph/I)^{-1}(\cat/I)\toto\FASinv{I}\sHom(I,\cat)=\CHOT{I}
  \quad\phantom{,}$$
and
$$\overline{\Theta'_I}:\CHOT{I}=\FASinv{I}\sHom(I,\cat)\toto(\FAsph/I)^{-1}(\cat/I)\quad,$$
induced by $\Theta_I$ and $\Theta'_I$ respectively, are equivalences of categories, quasi-inverse to each other.
\end{thm}

\begin{proof}
The equality $\FAsph/I=\Theta^{-1}_I(\FAS{I})$ follows from
proposition~\ref{3asphlocbase}, and the inclusion $\Theta'_I(\FAS{I})\subset\FAsph/I$
from the previous lemma. From this we deduce an adjoint pair of functors
$$\overline{\Theta}_I:(\FAsph/I)^{-1}(\cat/I)\toto\CHOT{I}\ \ ,\quad\ 
  \overline{\Theta'_I}:\CHOT{I}\toto(\FAsph/I)^{-1}(\cat/I)$$
the adjunction morphisms being induced by the adjunction morphisms
$$\e:\Theta_I\Theta'_I\toto1_{\sHom(I,\cat)}\quad,\qquad
  \eta:1_{\cat/I}\toto\Theta'_I\Theta_I\quad$$
(\emph{cf}.~\ref{3propadjGr}). Therefore, it is sufficient to show that for every functor $F:I\to\cat$,
the natural transformation $\e_F$ is in $\FAS{I}$ and that for every object $(A,v:A\to I)$
of $\cat/I$, the morphism $\eta_{(A,v)}$ of $\cat/I$ is in $\FAsph/I$.
\smallbreak

\emph{a}) By~\ref{3propadjGr}, for every object $i$ of
$I$, $\e_{F,i}$ is the morphism
$$\bigl(\Int F\bigr)\bigm/i\toto F(i)\quad$$
which sends an object $(j,a,p:j\to i)$ of $\bigl(\Int F\bigr)\!\bigm/\!i$, 
$\,j\in\ob(I)$, $a\in\ob(F(j))$, $p\in\fl(I)$, to the object $F(p)(a)$ of $F(i)$.
It is easy to check that this functor is a left adjoint to the inclusion functor
$$F(i)\toto\bigl(\Int F\bigr)\bigm/i\quad,\qquad a\longmapsto(i,a,1_i)\quad,$$
(\emph{cf}.~lemma~\ref{3carprecof}), which proves, by proposition~\ref{3adjgasph},
that it is aspheric. From this we deduce that $\e_F$ is componentwise aspheric, 
which proves the assertion relative to $\e_F$.
\smallbreak

\emph{b}) By~\ref{3propadjGr}, the morphism $\eta_{(A,v)}$ is the inclusion
$$A\toto\Int A/i\quad,\qquad a\longmapsto(v(a),a,1_{v(a)})\quad,$$
where, by an abuse of notation, $\int A/i$ stands for the cofibered category
over $I$ defined by the functor
$$I\toto\cat\quad,\qquad i\longmapsto A/i\quad.$$
It is easy to check that this functor is a left adjoint to the functor
$$\Int A/i\toto A\quad,\qquad(i,a,p:v(a)\to i)\longmapsto a\quad$$
(\emph{cf}.~proof of the lemma~\ref{3lemmeclef}), which proves that it is aspheric
(\ref{3adjgasph}) and finishes the proof of the theorem.
\end{proof}

\begin{paragr}\label{3defimdirh}
For every arrow $w:J\to I$ of $\cat$, we denote by
$$\imdir{w}:\sHom(J,\cat)\toto\sHom(I,\cat)\quad$$ 
the composite functor
$\imdir{w}=\Theta_I\circ\cat/w\circ\Theta'_J$.
$$
\xymatrixcolsep{2.8pc}
\xymatrix{
\sHom(J,\cat)\ar[r]^(.6){\Theta'_J}
&\cat/J\ar[r]^{\cat/w}
&\cat/I\ar[r]^(.4){\Theta_I}
&\sHom(I,\cat)
}$$
Since, by theorem~\ref{3theqcatGr} and the considerations of paragraph~\ref{3defasphrel}, we have 
$$\Theta'_J(\FAS{J})\subset\FAsph/J\, ,\ \ (\cat/w)(\FAsph/J)\subset\FAsph/I\, ,\ \ 
\Theta_I(\FAsph/I)\subset\FAS{I}\, ,$$
this functor induces a functor between the localized categories,
also denoted, by an abuse of notation, by
$$\imdirh{w}:\CHOT{J}\toto\CHOT{I}\quad,$$
which is the composite of the functors
$$
\xymatrixcolsep{2.8pc}
\xymatrix{
\CHOT{J}\ar[r]^(.35){\overline{\Theta'_J}}
&(\FAsph/J)^{-1}(\cat/J)\ar[r]^{\overline{\cat/w}}
&(\FAsph/I)^{-1}(\cat/I)\ar[r]^(.65){\overline{\Theta_I}}
&\CHOT{I}\ ,
}$$
induced by $\Theta'_J$, $\cat/w$ and $\Theta_I$. We have a commutative square
$$
\xymatrixcolsep{2.8pc}
\xymatrix{
\sHom(J,\cat)\ar[r]^{\imdir{w}}\ar[d]_{\gamma^{}_J}
&\sHom(I,\cat)\ar[d]^{\gamma^{}_I}
\\
\CHOT{J}\ar[r]_{\imdirh{w}}
&\CHOT{I}
}$$
whose vertical arrows are the localization functors. 
\end{paragr}

\begin{thm}\label{3imdirh}
For every morphism $w:J\to I$ of $\cat$, the functors
$$\imdirh{w}:\CHOT{J}\toto\CHOT{I}\quad,\qquad w^*:\CHOT{I}\toto\CHOT{J}\quad$$
form an adjoint pair.
\end{thm}

\begin{proof}
It follows from~\ref{3defhot},~\ref{3defasphrel},~\ref{3propadjGr} and from theorem~\ref{3theqcatGr}
that the pair of functors
$(\overline\Theta_I\circ\overline{\cat/w}\,,\,\overline{\Theta'_J}\circ w^*)$
$$\displaylines{
{\xymatrixcolsep{2.9pc}
\xymatrix{
(\FAsph/J)^{-1}(\cat/J)\ar[r]^{\overline{\cat/w}}
&(\FAsph/I)^{-1}(\cat/I)\ar[r]^(.65){\overline{\Theta}_I}
&\CHOT{I}
}
}\cr
\noalign{\vskip 3pt}
{\xymatrixcolsep{2.2pc}\hskip -.1cm
\xymatrix{
\CHOT{I}\ar[r]^{w^*}
&\CHOT{J}\ar[r]^(.33){\overline{\Theta'_J}}
&(\FAsph/J)^{-1}(\cat/J)
}
}
}$$
is an adjoint pair. By theorem~\ref{3theqcatGr}, the functors
$\overline{\Theta}_J$ and $\overline{\Theta'_J}$ are equivalences of
categories quasi-inverse to each other, therefore
$w^*\simeq\overline{\Theta}_J\circ(\overline{\Theta'_J}\circ w^*)$
is a right adjoint to
$(\overline{\Theta}_I\circ\overline{\cat/w})\circ\overline{\Theta'_J}=\imdirh{w}$,
which proves the theorem. 
\end{proof}

\begin{rem}
If $\Asph$ is the right asphericity structure associated with a
basic localizer $\W$ (\emph{cf}.~\ref{3strasphlocfond}), it follows from the theory
developed by \hbox{D.-C.~Cisinski~\cite{C1}} that the functor
$$\imdirh{w}:\CHOT{J}\toto\CHOT{I}\quad$$
is canonically isomorphic to the left derived functor of the left adjoint of the inverse image functor
$$w^*:\sHom(I,\cat)\toto\sHom(J,\cat)\quad.$$
\end{rem}

\section{Smooth functors}

\noindent
\emph{In this paragraph, we fix, once and for all, 
a right asphericity structure~$\Asph$.}

\begin{definition} \label{3defaiblisse}
A morphism $u:A\to B$ of $\cat$ is said to be
\emph{weakly \hbox{$\Asph$-smooth}}, or more simply 
\emph{weakly smooth}, if for every object $b$ of $B$, the canonical morphism
$$j_b:A_b\toto b\backslash A\quad,\qquad
  a\mapsto (a,1_b:b\to u(a))\ ,\quad a\in\ob(A_b)\ ,$$
is aspheric.
\end{definition}

\begin{exemple} \label{3exflisse}
A prefibration is a weakly smooth morphism.
Indeed, for every object $b$ of $B$, the functor $j_b$
has a right adjoint (dual of lemma~\ref{3carprecof}), and therefore it
is aspheric by proposition~\ref{3adjgasph}.
If $\Asph$ is the minimal right asphericity structure (\ref{3strasphmin}),
it follows from the characterization of aspheric functors for
this structure (\ref{3carfonctasphmin}) and from the dual of lemma~\ref{3carprecof}
that the weakly smooth morphisms with respect to the minimal right asphericity structure are exactly the prefibrations.
\end{exemple}

\begin{prop}\label{3carflisse}
Let $u:A\to B$ be a morphism of $\cat$. The following conditions are equivalent:
\begin{itemize}
\item[(a)] $u$ is weakly smooth;
\item[(b)] for every object $a$ of $A$, the fibers of the morphism
$$A/a\toto B/b\quad,\qquad b=u(a)\quad,$$
induced by $u$, are aspheric;
\item[(c)] for every diagram of cartesian squares
$$
\xymatrix{
&A''\ar[r]\ar[d]
&A'\ar[r]\ar[d]
&A\ar[d]^{u}
\\
&\smp{0}\ar[r]
&\smp{1}\ar[r]
&B
&\hskip -20pt,
\hskip 30pt}$$
where $\smp{0}\to\smp{1}$ stands for the inclusion $\{0\}\hookrightarrow\{0\to 1\}$, 
the morphism $A''\to A'$ is aspheric;
\item[(d)] for every arrow $g:b_0\to b_1$ of $B$, and every object
$a_1$ of $A_{b_1}$, the category $A(a_1,g)$, the objects of which are
the arrows $f:a\to a_1$ of $A$ whose target is $a_1$ and which lift $g$
\hbox{\emph{(i.e.~$u(f)=g$),}} and the morphisms of which are the commutative triangles in $A$
$$
\xymatrixrowsep{.6pc}
\xymatrixcolsep{2.6pc}
\xymatrix{
a\ar[rd]^{f}\ar[dd]_{h}
\\
&a_1
\\
a'\ar[ru]_{f'}
&
}$$
such that $h$ is a morphism of $A_{b_0}$ \emph{(i.e.~$u(h)=1_{b_0}$),} is aspheric.
\end{itemize}
\end{prop}

\begin{proof}
We leave it to the reader to check that for every arrow \hbox{$g:b_0\to b_1$}
of $B$, and every object $a_1$ of $A$ over $b_1$, the category
$A_{b_0}/(a_1,g)$ (defined by the functor $j_{b_0}:A_{b_0}\to b_0\backslash A$ 
and the object $(a_1,g:b_0\to u(a_1)=b_1)$ of $b_0\backslash A$),
as well as the fiber of $A/a_1\to B/b_1$ over the object $(b_0,g)$ of $B/b_1$,
are isomorphic to the category $A(a_1,g)$, which proves
the equivalence of conditions~(\emph{a}),~(\emph{b}) and~(\emph{d}).
Let us show the equivalence between~(\emph{c}) and~(\emph{d}). 
Observe that there is a one-to-one correspondence between arrows $g$ of $B$, as in~(\emph{d}), and functors $\smp{1}\to B$, as in~(\emph{c}). Using the notations of (\emph{c}), the inclusion $A''\to A'$ is, by definition, aspheric if and only if for every object $a'$ of $A'$, the category $A''/a'$ is aspheric. If $a'$ is an object of the fiber $A'_0\simeq A''$ of $A'$ over $0$, this is true without any hypothesis on $u$, for in that case $A''/a'$ has a final object. Therefore, it is sufficient to check this for $a'$ in the fiber $A'_1$ of $A'$ over $1$. In that case, $a'$ corresponds to an object $a_1$ of $A_{b_1}$, and we check that $A''/a'$ is isomorphic to the category $A(a_1,g)$ of~(\emph{d}). This completes the proof.
\end{proof}

\begin{cor}\label{3stchbaseflisse}
Weakly smooth morphisms are stable under base change,
\emph{i.e.}~for every cartesian square in $\cat$
$$
\xymatrix{
&A'\ar[r]\ar[d]_{u'}
&A\ar[d]^{u}
\\
&B'\ar[r]
&B
&\hskip -20pt,
}$$
if $u$ is weakly smooth, then so is $u'$.
\end{cor}

\begin{proof}
The corollary follows from the previous proposition, since
the condition ({\it c\/}) is stable under base change.
\end{proof}

\begin{prop}\label{3flisseloc}
Let $u:A\to B$ be a morphism of $\cat$. The following conditions
are equivalent:
\begin{itemize}
\item[(a)] $u$ is weakly smooth;
\item[(b)] for every object $a$ of $A$, the morphism $A/a\to B/b$,
$b=u(a)$, induced by $u$, is weakly smooth.
\end{itemize}
\end{prop}

\begin{proof}
Assume that the functor $u$ is weakly smooth, and let
$a$ be an object of $A$, and $b=u(a)$. By condition
(\emph{b}) of proposition~\ref{3carflisse}, in order to prove that
the functor $A/a\to B/b$, induced by $u$, is weakly smooth,
it is sufficient to show that for every object $(a'\!,f:a'\to a)$ of
$A/a$, the fibers of the functor $(A/a)/(a'\!,f)\to(B/b)/(b'\!,g)$,
where $(b'\!,g)=(u(a'),u(f))$, are aspheric. 
Now, this functor is canonically isomorphic to the functor
$A/a'\to B/b'$, induced by $u$,
the fibers of which are aspheric,
by condition ({\it b\/}) of proposition~\ref{3carflisse}, hence the assertion.
Conversely, assume that for every object $a$ of~$A$, the functor
$A/a\to B/b$, $b=u(a)$,
induced by $u$, is weakly smooth. Let us show
that, in that case, so is $u$. By condition (\emph{b}) of proposition~\ref{3carflisse},
the hypothesis that $A/a\to B/b$ is weakly smooth implies that the fibers of
the functor $(A/a)/(a,1_a)\to(B/b)/(b,1_b)$ are aspheric. 
Now, the latter is canonically isomorphic to the functor $A/a\to B/b$, induced by $u$,
which proves that $u$ is weakly smooth, 
by condition ({\it b\/}) of proposition~\ref{3carflisse}.
\end{proof}

\begin{definition}\label{3deflisse}
A morphism $u:A\to B$ of $\cat$ is said to be $\Asph$\nobreakdash-\emph{smooth},
or more simply \emph{smooth}, if for every diagram of cartesian squares
$$
\xymatrix{
&A''\ar[r]^{v}\ar[d]
&A'\ar[r]\ar[d]
&A\ar[d]^{u}
\\
&B''\ar[r]_{w}
&B'\ar[r]
&B
&\hskip -20pt,
}$$
if the morphism $w$ is aspheric, then so is $v$.
\end{definition}

\begin{prop}\label{3sorlisse}
The class of smooth morphisms is stable under composition and base change. 
\end{prop}

\begin{proof}
It is a formal consequence of the definition.
\end{proof}

\begin{prop}\label{3lisseflisse}
A smooth morphism is weakly smooth. 
\end{prop}

\begin{proof}
It follows from condition (\emph{c}) of proposition~\ref{3carflisse}.
\end{proof}

\begin{prop}\label{3isoloclisse}
A local isomorphism is a smooth morphism.
\end{prop}

\begin{proof}
Let $u:A\to B$ be a local isomorphism, \emph{i.e.}~an arrow
$u:A\to B$ of $\cat$ such that, for every object $a$ of $A$, the functor
$A/a\to B/b$, $b=u(a)$, induced by $u$, is an isomorphism,
and consider the diagram of cartesian squares
$$
\xymatrix{
&A''\ar[r]^{v'}\ar[d]_{u''}
&A'\ar[r]^{v}\ar[d]^{u'}
&A\ar[d]^{u}
\\
&B''\ar[r]_{w'}
&B'\ar[r]_w
&B
&\hskip -20pt,
}$$
where we assume that the morphism $w'$ is aspheric.
For every object $a'$ of $A'$, we deduce a diagram of cartesian squares
$$
\xymatrix{
&A''/a'\ar[r]\ar[d]
&A'/a'\ar[r]\ar[d]
&A/a\ar[d]
\\
&B''/b'\ar[r]
&B'/b'\ar[r]
&B/b
&\hskip -20pt,
}$$
where $a=v(a')$, $b'=u'(a')$, $b=u(a)=w(b')$, the vertical arrows
of which are isomorphisms. The category $B''/b'$ being aspheric by hypothesis,
so is $A''/a'$,
which proves the proposition.
\end{proof}

\begin{prop}\label{3loclisse}
Let $u:A\to B$ be a morphism of $\cat$. The following conditions are
equivalent:
\begin{itemize}
\item[(a)] $u$ is smooth;
\item[(b)] for every object $a$ of $A$, the morphism $A/a\to B/b$,
$b=u(a)$, induced by $u$, is smooth.
\end{itemize}
\end{prop}

\begin{proof}
Assume that $u$ is smooth, and let $a$ be an object of $A$,
$b=u(a)$. Stability of smooth morphisms under
base change (\ref{3sorlisse}) implies that the functor $A/b\to B/b$,
induced by $u$, is smooth. The canonical functor $A/a\to A/b$ being a local isomorphism, 
it follows from the previous proposition that
it is smooth. Stability of smooth morphisms under composition (\ref{3sorlisse}) thus
implies that the composite $A/a\to B/b$ is smooth.
\smallbreak

Conversely, assume that for every object $a$ of $A$, the functor 
$$A/a\toto B/b\ ,\qquad b=u(a)\ , $$
induced by $u$, is smooth, and consider a diagram of cartesian squares
$$
\xymatrix{
&A''\ar[r]^{v}\ar[d]
&A'\ar[r]\ar[d]
&A\ar[d]^{u}
\\
&B''\ar[r]_{w}
&B'\ar[r]
&B
&\hskip -20pt,
}$$
where $w$ is an aspheric functor. For every object $a$ of $A$, we
deduce a diagram of cartesian squares
$$
\xymatrix{
&A''/a\ar[r]\ar[d]
&A'/a\ar[r]\ar[d]
&A/a\ar[d]
\\
&B''/b\ar[r]
&B'/b\ar[r]
&B/b
&\hskip -20pt,
}$$
where $b=u(a)$. The functor $w$ being aspheric, so is
the morphism $B''/b\to B'/b$ (\ref{3asphlocbase}). The functor
$A/a\to B/b$ being smooth, it follows that $A''/a\to A'/a$ is aspheric.
Since this is true for every object $a$ of $A$, it follows from
proposition~\ref{3asphlocbase} that $v$ is aspheric, 
which finishes the proof.
\end{proof}

\begin{cor}
Locally aspheric functors are stable under smooth base change,
\emph{i.e.}~for every cartesian square in $\cat$
$$
\xymatrix{
&A'\ar[r]^{v}\ar[d]_{u'}
&A\ar[d]^{u}
\\
&B'\ar[r]_{w}
&B
&\hskip -20pt ,
}$$
if $u$ is smooth and $w$ is locally aspheric, then $v$ is locally aspheric.
\end{cor}

\begin{proof}
For every object $a'$ of $A'$, we have a cartesian square in $\cat$
$$
\xymatrix{
&A'/a'\ar[r]\ar[d]
&A/a\ar[d]
\\
&B'/b'\ar[r]
&B/b
&\hskip -20pt ,
}$$
where $a=v(a')$, $b'=u'(a')$, and $b=u(a)=w(b')$. By the previous proposition,
if $u$ is smooth, then so is $A/a\to B/b$, and if  $w$ is locally aspheric,
then the functor $B'/b'\to B/b$ is aspheric, which proves that $A'/a'\to A/a$ is aspheric
and finishes the proof.
\end{proof}

\begin{paragr}\label{3cleftechn}
Let
$$
\xymatrix{
&A\ar[d]^{u}
\\
B'\ar[r]_{w}
&B
}$$
be a diagram in $\cat$. We can form the cartesian square
$$
\xymatrix{
&A'\ar[r]^{v}\ar[d]_{u'}
&A\ar[d]^{u}
\\
&B'\ar[r]_{w}
&B
&\hskip -20pt,
}$$
where $A'=B'\times_BA$. We can also form the ``$2$\nobreakdash-square''
$$
\UseAllTwocells
\xymatrix{
&A_0'\ar[r]^{v_0}\ar[d]_{u_0'}
&A\ar[d]^{u}
\\
&B'\ar[r]_{w}
&B\ulcompositemap<\omit>{\alpha}
&\hskip -20pt,
}$$
where $A'_0$ is the comma category whose objects are triples
$$(b',\,a,\,g:w(b')\to u(a))\quad ,\qquad b'\in\ob(B')\ ,\quad a\in\ob(A)\ ,\quad g\in\fl(B)\ ,$$
a morphism from $(b'_0,a_0,g_0)$ to $(b'_1,a_1,g_1)$ being a pair
$(g',f)$, where $g':b'_0\to b'_1$ is an arrow of $B'$, and $f:a_0\to a_1$ is an arrow of $A$,
such that the diagram
$$
\xymatrixcolsep{2.8pc}
\xymatrix{
w(b'_0)\ar[r]^{w(g')}\ar[d]_{g_0}
&w(b'_1)\ar[d]^{g_1}
\\
u(a_0)\ar[r]_{u(f)}
&u(a_1)
}$$
is commutative. The functors $u'_0$, $v_0$ are defined by
$$\begin{aligned}
&{u'_0(b',a,g)=b'\ ,\quad v_0(b',a,g)=a\ ,\qquad(b',a,g)\in\ob(A'_0)\ ,}\\
&{u'_0(g',f)=g'\ ,\hskip 18pt v_0(g',f)=f\ ,\hskip 27pt (g',f)\in\fl(A'_0)\ ,}
\end{aligned}$$
and the natural transformation $\alpha:wu'_0\to uv_0$ is defined by
$$\alpha_{(b',a,g)}=g:wu'_0(b',a,g)=w(b')\toto u(a)=uv_0(b',a,g)\ ,\quad(b',a,g)\in\ob(A'_0)\quad.$$
We fix an object $b'_0$ of $B'$, an object $a_1$ of $A$ and a
morphism $g:w(b'_0)\to u(a_1)$. We set $b_0=w(b'_0)$, $b_1=u(a_1)$.
We shall associate to these data three categories $C_0$, $C_1$, $C_2$,
and leave it to the reader to check that they are isomorphic.
\smallbreak

\emph{a}) \emph{Definition of $C_0$.}
The triple $(b'_0,a_1,g)$ is an object of $A'_0$, and
there is a canonical functor $A'\to A'_0$ which sends an object
$(b',a)$ of $A'$, $b'\in\ob(B')$, \hbox{$a\in\ob(A)$}, $w(b')=u(a)$,
to the object $(b',a,1_{u(a)})$ of $A'_0$. The category $C_0$ is the
category $A'/(b'_0,a_1,g)$.
\smallbreak

\emph{b}) \emph{Definition of $C_1$.}
The pair $(b'_0,g)$ is an object of $B'/b_1$, and there is a cartesian square
$$
\xymatrix{
&A'/a_1\ar[r]\ar[d]
&A/a_1\ar[d]
\\
&B'/b_1\ar[r]
&B/b_1
&\hskip -20pt.
}$$
The category
$C_1$ is the category $(A'/a_1)/(b'_0,g)$, defined by the left vertical arrow.
\smallbreak

\emph{c}) \emph{Definition of $C_2$.}
Let us consider the category $(B'/b'_0)^*$ obtained by adding a new final object to $B'/b'_0$, and the inclusion $B'/b'_0\hookrightarrow(B'/b'_0)^*$.
By the universal property of this construction, there
is a unique morphism $(B'/b'_0)^*\to B/b_1$ of $\cat$ such
that the image of the final object of $(B'/b'_0)^*$ under this functor  is
the final object $(b_1,1_{b_1})$ of $B/b_1$, and such that the square
$$\xymatrix{
&B'/b'_0\ar[r]\ar@{^{(}->}[d]
&B/b_0\ar[d]
\\
&(B'/b'_0)^*\ar[r]
&B/b_1
&\hskip -20pt
}$$
is commutative, where the horizontal upper arrow is induced by $w$, 
and the vertical arrow on the right is defined by $g$.
Let us consider the composite 
$$(B'/b'_0)^*\toto B/b_1\toto B\quad$$ 
of this functor
with the forgetful functor from $B/b_1$ to $B$,
and let us form the diagram of cartesian squares
$$
\xymatrix{
&A''\ar[d]\ar[r]
&\overline{A'}\ar[d]\ar[r]
&A\ar[d]
\\
&B'/b'_0\ar[r]
&(B'/b'_0)^*\ar[r]
&B
&\hskip -20pt.
}
$$
The category $C_2$ is the category $A''/a'_1$, where
$a'_1$ is the object of $\overline{A'}=(B'/b'_0)^*\times_BA$
which is over the final object of $(B'/b'_0)^*$ and whose projection 
in $A$ is $a_1$.
\end{paragr}

\begin{thm}\label{3carlisse}
Let $u:A\to B$ be an arrow of $\cat$. The following conditions
are equivalent:
\begin{itemize}
\item[(a)] $u$ is smooth;
\item[(b)] for every diagram of cartesian squares
$$
\xymatrix{
&A''\ar[r]^{v}\ar[d]
&A'\ar[r]\ar[d]
&A\ar[d]^{u}
\\
&B''\ar[r]_{w}
&B'\ar[r]
&B
&\hskip -20pt,
}$$
if the functor $w$ is aspheric with respect to the minimal right asphericity structure, \emph{i.e.}~if it has a right adjoint, then the morphism $v$ is aspheric;
\item[(c)] for every diagram of cartesian squares
$$
\xymatrix{
&A''\ar[r]^{v}\ar[d]
&A'\ar[r]\ar[d]
&A\ar[d]^{u}
\\
&B''\ar[r]_{w}
&B'\ar[r]
&B
&\hskip -20pt,
}$$
where $B''$ is a category which has a final object, $B'$ is obtained by adding
a new final object to $B''$, and $w$ is the canonical inclusion, the morphism $v$ is aspheric;
\item[(d)] for every diagram of cartesian squares
$$
\xymatrix{
&A''\ar[r]^{}\ar[d]_{u''}
&A'\ar[r]\ar[d]^{u'}
&A\ar[d]^{u}
\\
&B''\ar[r]_{}
&B'\ar[r]
&B
&\hskip -20pt,
}$$
and for every object $a'$ of $A'$, the morphism
$$A''/a'\toto B''/b'\quad,\qquad b'=u'(a')\quad,$$
induced by $u''$, is aspheric.
\end{itemize}
\end{thm}

\begin{proof}
The implications (\emph{a}) $\Rightarrow$ (\emph{b}) $\Rightarrow$ (\emph{c})
are clear. Let us show the implication (\emph{c}) $\Rightarrow$ (\emph{d}).
As the condition (\emph{c}) is stable under base change, it
is sufficient to show that if
$$
\xymatrix{
A'\ar[r]^{v}\ar[d]_{u'}
&A\ar[d]^{u}
\\
B'\ar[r]_{w}
&B
}$$
is a cartesian square, $a_1$ an object of $A$, $b_1=u(a_1)$, and
$(b'_0,\,g:w(b'_0)\to b_1)$ an object of $B'/b_1$, then the
category $(A'/a_1)/(b'_0,g)$ is aspheric. By~\ref{3cleftechn},
this category is isomorphic to the category $A''/a'_1$, where
$$
\xymatrix{
&A''\ar[d]\ar[r]
&\overline{A'}\ar[d]\ar[r]
&A\ar[d]
\\
&B'/b'_0\ar[r]
&(B'/b'_0)^*\ar[r]
&B
&
}
$$
is the diagram of cartesian squares considered in
\ref{3cleftechn}, (\emph{c}), and $a'_1$ is the object of $\overline{A'}$
over the final object of $(B'/b'_0)^*$ whose image in $A$ is
the object $a_1$. Now, by condition (\emph{c}), the
functor $A''\to\overline{A'}$ is aspheric. From this we deduce that
the category $A''/a'_1$ is aspheric, therefore so is the category $(A'/a_1)/(b'_0,g)$. 
\smallbreak

It remains to show the implication (\emph{d}) $\Rightarrow$ (\emph{a}).
As the condition (\emph{d}) is stable under base change, it is sufficient
to show that if
$$
\xymatrix{
A'\ar[r]^{v}\ar[d]_{u'}
&A\ar[d]^{u}
\\
B'\ar[r]_{w}
&B
}$$
is a cartesian square, where $w$ is an aspheric morphism, then $v$
is aspheric, too.\break 
Let~$a$ be an object of $A$. We need
to show that the category $A'/a$ is aspheric. Let $b=u(a)$.
By condition (\emph{d}), the morphism $A'/a\to B'/b$,
induced by $u'$, is aspheric. By hypothesis, the functor
$w$ is aspheric, hence the category $B'/b$ is aspheric, and
therefore so is $A'/a$, which concludes the proof.
\end{proof}

\begin{cor}\label{3lisselocasph}
A smooth functor is locally aspheric.
\end{cor}

\begin{proof}
It is an immediate consequence of condition (\emph{d}) of the previous theorem.
\end{proof}

\begin{prop}\label{3carfibr}
Let $u:A\to B$ be a morphism of $\cat$. The following conditions are
equivalent:
\begin{itemize}
\item[(a)] u is a fibration;
\item[(b)] u is smooth with respect to the minimal right asphericity structure;
\item[(c)] for every diagram of cartesian squares
$$
\xymatrix{
&A''\ar[r]^{v}\ar[d]
&A'\ar[r]\ar[d]
&A\ar[d]^{u}
\\
&B''\ar[r]_{w}
&B'\ar[r]
&B
&\hskip -20pt,
}$$
if the functor $w$ has a right adjoint, then so has $v$;
\item[(d)] for every diagram of cartesian squares
$$
\xymatrix{
&A''\ar[r]\ar[d]
&A'\ar[r]\ar[d]
&A\ar[d]^{u}
\\
&\smp{1}\ar[r]
&\smp{2}\ar[r]
&B
&\hskip -20pt,
\hskip 30pt}$$
where $\smp{1}\to\smp{2}$ stands for the inclusion $\{0\to 1\}\hookrightarrow\{0\to 1\to2\}$, 
the functor $A''\to A'$ has a right adjoint.
\end{itemize}
\end{prop}

\begin{proof}
Since aspheric functors with respect to the minimal right asphericity
structure are exactly those which have a right adjoint (\ref{3carfonctasphmin}),
equivalence between conditions (\emph{b}) and (\emph{c}) is tautological.
As the implication (\emph{c}) $\Rightarrow$ (\emph{d}) is clear,
it is sufficient to prove the implications (\emph{a}) $\Rightarrow$ (\emph{b})
and (\emph{d}) $\Rightarrow$ (\emph{a}).
\smallbreak

Since fibrations are stable under base change, 
in order to show the implication (\emph{a}) $\Rightarrow$ (\emph{b})
it is sufficient to show that for every cartesian square
$$
\xymatrixcolsep{2.3pc}
\xymatrix{
A'=B'\times_BA\ar[r]\ar@<2.3ex>[d]
&A\ar[d]^{u}
\\
\hskip 22pt B'\ar[r]_(.6){w}
&B
&\hskip -20pt,
}$$
if $u$ is a fibration, and if for every object $b$ of $B$, the
category $B'/b$ has a final object, then for every object $a$
of $A$, the category $A'/a$ has a final object.
\smallbreak

Accordingly, let $a_1$ be an object of $A$, $b_1$ its image in $B$,
and $(b'_0,\,g:w(b'_0)\to b_1)$ a final object of $B'/b_1$. Since
$u$ is a fibration, there exists a hypercartesian morphism
$k:a_0\to a_1$ of $A$
such that $u(k)=g$. We
shall show that $\bigl((b'_0,a_0),\,k:a_0\to a_1\bigr)$ is a final object of $A'/a_1$.
Let $\bigl((b',a),\,f:a\to a_1\bigr)$ be an object of $A'/a_1$, $w(b')=u(a)$.
We have to show that there exists an arrow of $A'/a_1$ whose domain is
$\bigl((b',a),f\bigr)$ and whose codomain is $\bigl((b'_0,a_0),k\bigr)$, and that it is unique. In other words, 
we have to show that there exists a unique pair $(g',h)$, where \hbox{$g':b'\to b'_0$} is an
arrow of $B'$ and $h:a\to a_0$ is an arrow of $A$, such that
$w(g')=u(h)$ and $f=kh$. For such a pair, we have $u(f)=u(k)u(h)=gw(g')$,
which means that 
$$g':\bigl(b',\,u(f):w'(b')\to b_1\bigr)\toto\bigl(b'_0,\,g:w(b'_0)\to b_1\bigr)\quad$$
is an arrow of $B'/b_1$. 
Now, the hypothesis that $(b'_0,g)$ is a
final object of $B'/b_1$ implies that such a $g'$ exists and is unique. 
Since the morphism $k$ is hypercartesian, there exists a unique
arrow \hbox{$h:a\to a_0$} of $A$ such that $u(h)=w(g')$ and
$f=kh$, which proves the assertion.
\smallbreak

It remains to show the implication (\emph{d}) $\Rightarrow$ (\emph{a}).
Let $\smp{2}\to B$ be a functor defined by a pair of composable arrows
$$\xymatrixcolsep{2pc}\xymatrix{b_0\ar[r]^{g_0}&b_1\ar[r]^{g_1}&b_2}$$
of $B$. Let us form the diagram of cartesian squares which appears in the statement of condition (\emph{d}).
This condition means that for every object $a'_2$ of $A'$ whose 
image $a_2$ in $A$ is over $b_2$, the category
$A''/a'_2$ has a final object. Let us describe this category.
The set of objects of $A''/a'_2$ can be canonically identified with the disjoint sum
$A''_0{\scriptstyle\coprod} A''_1$, where
$$\begin{aligned}
&{A''_0=\{(a_0,f_0)\mid a_0\in\ob(A),\,u(a_0)=b_0,\,f_0:a_0\to a_2\in\fl(A),\,u(f_0)=g_1g_0\}}\quad,\\
&{A''_1=\{(a_1,f_1)\mid a_1\in\ob(A),\,u(a_1)=b_1,\,f_1:a_1\to a_2\in\fl(A),\,u(f_1)=g_1\}}\quad,\\
\end{aligned}$$
and for every $(a_0,f_0),\,(a'_0,f'_0)\in A''_0$, $(a_1,f_1),\,(a'_1,f'_1)\in A''_1$,
we have
$$\begin{aligned}
&{\Hom_{A''/a'_2}((a_0,f_0),(a'_0,f'_0))=\{g\mid g:a_0\to a'_0\in\fl(A),\,u(g)=1_{b_0},\,f_0=f'_0g\}\quad,}\\
&{\Hom_{A''/a'_2}((a_0,f_0),(a_1,f_1))=\{g\mid g:a_0\to a_1\in\fl(A),\,u(g)=g_0,\,f_0=f_1g\}\quad,}\\
&{\Hom_{A''/a'_2}((a_1,f_1),(a_0,f_0))=\varnothing\quad,}\\
&{\Hom_{A''/a'_2}((a_1,f_1),(a'_1,f'_1))=\{g\mid g:a_1\to a'_1\in\fl(A),\,u(g)=1_{b_1},\,f_1=f'_1g\}\quad.}
\end{aligned}$$
Since the category $A''/a'_2$ has a final object, it
is non-empty. The special case $b_0=b_1$ and
$g_0=1_{b_1}$ then shows that for every arrow $b_1\to b_2$ of
$B$ and for every object $a_2$ of $A$ over $b_2$, there exists a
morphism $a_1\to a_2$ of $A$ over $b_1\to b_2$. Coming back
to the general case
(\smash{$\xymatrixcolsep{1.3pc}\xymatrix{b_0\ar[r]^{g_0}&b_1\ar[r]^{g_1}&b_2}$}
being arbitrary), this implies that $A''_0$ and $A''_1$ are
non-empty sets. Since there is no morphism in
$A''/a'_2$ from an object of $A''_1$ to an object
of $A''_0$, we deduce that the final object $(a_1,f_1)$
of $A''/a'_2$ belongs to $A''_1$. This implies that
$(a_1,f_1)$ is also a final object of the full subcategory 
of $A''/a'_2$ whose objects are the objects of $A''_1$. This means
exactly that $f_1$ is a cartesian morphism
over $g_1$ and shows that $f_1$ is determined by the sole $g_1$
and does not depend on $g_0$. As $g_1$ is an arbitrary arrow of $B$
and $a_2$ an arbitrary object of the fiber of $A$ over the target
of $g_1$, this implies already that $u$ is a prefibration.
Now, $(a_1,f_1)$ is a final object of $A''/a'_2$. Therefore,
for every object $(a_0,f_0)$ of $A''/a'_2$ which belongs to the set
$A''_0$, $f_0:a_0\to a_2\in\fl(A)$, $u(f_0)=g_1g_0$,
there exists a unique morphism $g:(a_0,f_0)\to(a_1,f_1)$ of $A''/a'_2$.
In other words, there is a unique arrow $g:a_0\to a_1$ of $A$ such that
$f_1g=f_0$ and $u(g)=g_0$. Since $g_0$ is an arbitrary arrow
of $B$ whose codomain is the domain of $g_1$, and since $f_1$ does not depend
on $g_0$, it follows that the morphism
$f_1$ is hypercartesian, which finishes the proof.
\end{proof}

\begin{cor}
Fibrations are smooth functors (with respect to any right asphericity structure).
\end{cor}

\begin{proof}
The corollary follows from condition (\emph{c}) of the
previous proposition, condition (\emph{b}) of theorem~\ref{3carlisse},
and from proposition~\ref{3adjgasph}.
\end{proof}

\begin{exemple}
If $\Asph$ is the right asphericity structure associated with a basic localizer
 $\W$ (example~\ref{3strasphlocfond}), 
Grothendieck's characterization of $\W$\nobreakdash-smooth functors
and theorem~\ref{3carlisse} imply that, if $u$ is a morphism of $\cat$, the following conditions
are equivalent:
\begin{itemize}
\item[(\emph{a})] $u$ is $\W$\nobreakdash-smooth;
\item[(\emph{b})] $u$ is weakly $\Asph$\nobreakdash-smooth;
\item[(\emph{c})] $u$ is $\Asph$\nobreakdash-smooth.
\end{itemize}
On the other hand, since there exist prefibrations which are not fibrations, 
proposition~\ref{3carfibr} and example~\ref{3exflisse}
show that if $\Asph$ is the minimal right asphericity structure,
then the class of $\Asph$\nobreakdash-smooth functors is a proper
subclass of the class of weakly $\Asph$\nobreakdash-smooth functors.
\end{exemple}

\section{Smooth functors and base change morphisms}

\begin{lemme}\label{3cartint}
Let $w:J\to I$ be a morphism of $\cat$ and $F:I\to\cat$ be a functor.
We then have a cartesian square
$$
\xymatrix{
&\int Fw\ar[r]^{\ind{w}{F}}\ar[d]_{\Cofibr{F\hskip -1pt w}}
&\int F\ar[d]^{\Cofibr{F}}
\\
&J\ar[r]_{w}
&I
&\hskip -20pt,
}$$
where $\Cofibr{Fw}$ and $\Cofibr{F}$ stand for the cofibrations associated with the functors
$Fw$ and $F$ respectively,
and $\ind{w}{F}$ stands for the functor 
$(j,a)\mapsto (w(j),a)$, $(j,a)\in\ob(\int Fw)$,
induced by $w$. 
\end{lemme}

\begin{proof}
The lemma follows from an easy verification which is left to the reader.
\end{proof}

\begin{paragr}\label{3morchbase}
Let
$$
\mathcal{D}\,=\quad
\raise 20pt\vbox{
\xymatrixrowsep{1.8pc}
\xymatrix{
A'\ar[r]^{w}\ar[d]_(.47){u'}
&A\ar[d]^{u}
\\
B'\ar[r]_{v}
&B
}
}
\qquad\qquad
$$
be a cartesian square of $\cat$. For every functor $F:A\to\cat$,
we deduce a composite cartesian square
$$
\xymatrix{
\int Fw\ar[r]^{\ind{w}{F}}\ar[d]_{\Cofibr{Fw}}
&\int F\ar[d]^{\Cofibr{F}}
\\
A'\ar[r]^{w}\ar[d]_{u'}
&A\ar[d]^{u}
\\
B'\ar[r]_{v}
&B
}$$
(\emph{cf.}~\ref{3cartint}). For every object $b'$ of $B'$, the functor $\ind{w}{F}$
induces a functor
$$(\textstyle\int Fw)/b'\toto(\textstyle\int F)/v(b')\quad,$$
and we notice that
$$(\textstyle\int Fw)/b'=(\ctlhprim{u}w^*(F))(b')\qquad\hbox{and}\qquad
(\textstyle\int F)/v(b')=(v^*\ctlh{u}(F))(b')\quad$$
(\emph{cf.}~\ref{3defimdirh}). We deduce a morphism
$$\kappa^{}_{\mathcal{D}}:\ctlhprim{u}w^*\toto v^*\ctlh{u}\quad$$
of $\sHom\bigl(\sHom(A,\cat),\sHom(B',\cat)\bigr)$,
called {\it base change morphism associated with
the square\/} $\mathcal{D}$.
\end{paragr}

\begin{prop}\label{3chbaselisse}
Let
$$
\mathcal{D}\,=\quad
\raise 20pt\vbox{
\xymatrixrowsep{1.8pc}
\xymatrix{
A'\ar[r]^{w}\ar[d]_(.47){u'}
&A\ar[d]^{u}
\\
B'\ar[r]_{v}
&B
}
}
\qquad\qquad
$$
be a cartesian square of $\cat$, where $v$ is a smooth functor.
Then the base change morphism
$\kappa^{}_{\mathcal{D}}:\ctlhprim{u}w^*\to v^*\ctlh{u}$ is componentwise aspheric. 
In other words, for every functor $F:A\to\cat$ and for every object $b'$ of $B'$, the morphism
$$\kappa^{}_{\mathcal{D},F}(b'):(\ctlhprim{u}w^*(F))(b')\toto(v^*\ctlh{u}(F))(b')\quad$$
is aspheric.
\end{prop}

\begin{proof}
Let $F:A\to\cat$ be a functor, and consider
the cartesian square
$$
\xymatrix{
&\int Fw\ar[r]^{\ind{w}{F}}\ar[d]_{u'\Cofibr{Fw}}
&\int F\ar[d]^{u\Cofibr{F}}
\\
&B'\ar[r]_{v}
&B
&\hskip -20pt.
}$$
Since $v$ is a smooth functor, it follows from condition
(\emph{d}) of theorem~\ref{3carlisse}
that for every object $b'$ of $B'$, the functor 
$\kappa^{}_{\mathcal{D},F}(b'):(\int Fw)/b'\to(\int F)/v(b')$, induced by $\ind{w}{F}$,
is aspheric, which proves the proposition.
\end{proof}

\begin{thm}\label{3carlisseder}
Let $u:A\to B$ be a morphism of $\cat$. The following conditions
are equivalent~:
\begin{itemize}
\item[(a)] $u$ is smooth;
\item[(b)] for every diagram of cartesian squares in $\cat$
$$
\xymatrix{
&A''\ar[r]\ar[d]
&B''\ar[d]
\\
&A'\ar[r]\ar[d]
&B'\ar[d]
\\
&A\ar[r]_u
&B
&\hskip -20pt,
}$$
the base change morphism associated with the upper square
is componentwise an aspheric functor.
\end{itemize}
\end{thm}

\begin{proof}
Implication (\emph{a}) $\Rightarrow$ (\emph{b}) follows from
proposition~\ref{3chbaselisse} and from stability of
smooth morphisms under base change (proposition~\ref{3sorlisse}).
Conversely, we notice that condition (\emph{b}), applied
to the constant functor $B''\to\cat$ with value the final 
category $e$, implies that for every object $a'$ of $A'$,
if we denote its image in $B'$ by $b'$, the morphism $A''/a'\to B''/b'$
is aspheric. We deduce that $u$ fulfils condition
(\emph{d}) of theorem~\ref{3carlisse}, which proves the theorem.
\end{proof}

\begin{rem} 
Let
$$
\mathcal{D}\,=\quad
\raise 20pt\vbox{
\xymatrixrowsep{1.8pc}
\xymatrix{
A'\ar[r]^{w}\ar[d]_(.47){u'}
&A\ar[d]^{u}
\\
B'\ar[r]_{v}
&B
}
}
\qquad\qquad
$$
be a cartesian square of $\cat$, and
$$
\xymatrixrowsep{.7pc}
\xymatrixcolsep{2.9pc}
\UseAllTwocells
\xymatrix{
\sHom(A,\cat)\ddrcompositemap<\omit>{\hskip 7pt\kappa_{\mathcal{D}}}\ar[r]^{w^*}\ar[dd]_{u^{}_!}
&\sHom(A',\cat)\ar[dd]^{u'_!}
\\
&&u'_!w^*\ar[r]^{\kappa_{\mathcal{D}}}
&v^*u^{}_!
\\
\sHom(B,\cat)\ar[r]_{v^*}
&\sHom(B',\cat)
}$$
the base change morphism (\ref{3morchbase}).
One checks easily that the natural transformation
$$
\xymatrixrowsep{.7pc}
\xymatrixcolsep{2.9pc}
\UseAllTwocells
\xymatrix{
\CHOT{A}\ddrcompositemap<\omit>{\hskip 7pt\overline{\kappa}_{\mathcal{D}}}\ar[r]^{w^*}\ar[dd]_{u^{}_!}
&\CHOT{A'}\ar[dd]^{u'_!}
\\
&&u'_!w^*\ar[r]^(.45){\overline{\kappa}_{\mathcal{D}}}
&v^*u^{}_!\quad,
\\
\CHOT{B}\ar[r]_{v^*}
&\CHOT{B'}
}$$
induced by $\kappa_{\mathcal{D}}$ by localization, is
the ``base change morphism'' formally defined by
the adjunctions (\emph{cf.}~theorem~\ref{3imdirh}): the natural transformation $\overline{\kappa}_{\mathcal{D}}$
is the composite
$$
\xymatrixcolsep{2.9pc}
\xymatrix{
u'_!w^*\ar[r]^(.3){u'_!w^*\star\eta}
&u'_!w^*u^*u^{}_!=u'_!u'{}^*v^*u^{}_!\ar[r]^(.7){\varepsilon'\star v^*u^{}_!}
&v^*u^{}_!
}\quad,$$
where
$$\eta:1_{\CHOT{A}}\toto u^*u^{}_!\quad,\qquad\varepsilon':u'_!u'{}^*\toto1_{\CHOT{B'}}$$
denote the adjunction morphisms. Proposition~\ref{3chbaselisse}
implies that if the functor $v$ is smooth, then the natural transformation $\overline{\kappa}_{\mathcal{D}}$
is an isomorphism.
\end{rem}


\begin{thebibliography}{10}

\bibitem{SGA4}
M.~Artin{, A. Grothendieck, J.-L. Verdier}.
\newblock {\em Th\'eorie des topos et cohomologie \'etale des sch\'emas
  \emph{(SGA4)}}.
\newblock Lecture Notes in Mathematics, Vol. 269, 270, 305. Springer-Verlag,
  1972-1973.

\bibitem{Br}
K.~S. Brown.
\newblock Abstract homotopy and generalized sheaf cohomology.
\newblock {\em Transactions of the {A}mer. {M}ath. {S}oc.}, 186:\ 419--458,
  1973.

\bibitem{Chiche}
Jonathan Chiche.
\newblock Structures d'asph\'ericit\'e \`a droite.
\newblock Master's thesis, under the supervision of G.~Maltsiniotis and
  P.~A.~Melli\`es, 2009.

\bibitem{C1}
D.-C. Cisinski.
\newblock Les pr\'efaisceaux comme mod\`eles des types d'homotopie.
\newblock {\em Ast\'erisque}, 308:\break 1--392, 2006.

\bibitem{C2}
D.-C. Cisinski\phantom{}.
\newblock Le localisateur fondamental minimal.
\newblock {\em Cahiers de topologie et g\'eom\'etrie diff\'erentielle
  cat\'egoriques}, 45-2:\ 109--140, 2004.

\bibitem{PS}
A.~Grothendieck.
\newblock \emph{Pursuing stacks}.
\newblock Manuscript, 1983, to be published in \emph{Documents
  Ma\-th\'e\-ma\-tiques}.

\bibitem{Der}
A.~Grothendieck.
\newblock \emph{Les d\'erivateurs}.
\newblock Manuscript, 1990,~www.math.jussieu.fr/\raise -3.3pt\vbox{\hbox{$\widetilde{ \ }\,$}}maltsin/groth /Derivateurs.html.

\bibitem{Hel}
A.~Heller.
\newblock Homotopy theories.
\newblock {\em Memoirs of the Amer. Math. Soc.}, 71(383), 1988.

\bibitem{Ma}
G.~Maltsiniotis.
\newblock La th\'eorie de l'homotopie de {G}rothendieck.
\newblock {\em Ast\'erisque}, 301:\ 1--140, 2005.

\bibitem{Qu}
D.~Quillen.
\newblock {\em \emph{Higher algebraic K-theory: I}, {Algebraic K-theory I}},
  pages 85--147.
\newblock Lecture Notes in Mathematics, Vol. 341. Springer-Verlag, 1973.

\bibitem{Th}
R.~W. Thomason.
\newblock {$\cat$} as a closed model category.
\newblock {\em Cahiers de topologie et g\'eom\'etrie diff\'erentielle
  cat\'egoriques}, 21-3:\ 305--324, 1980.

\end{thebibliography}
\end{document}